\documentclass[11pt,reqno]{amsart}

\usepackage[utf8]{inputenc}
\usepackage{setspace}
\usepackage{geometry}
\usepackage{enumerate}
\usepackage{enumitem, xcolor, amssymb,latexsym,amsmath,bbm}
\usepackage{mathtools}
\usepackage[mathscr]{euscript}
\usepackage{amsmath}
\usepackage{amssymb}
\usepackage{enumitem}

\usepackage{tikz}
\usepackage{wrapfig}
\usepackage{enumerate}
\usepackage{graphicx}
\usepackage{subfigure}
\usepackage{esint}
\usetikzlibrary{arrows.meta}
\usepackage{mathrsfs}
\usetikzlibrary{decorations.markings}

\usepackage[colorlinks=true,citecolor=blue,
linkcolor=blue,urlcolor=blue]{hyperref}

\newtheorem{theorem}{Theorem}
\newtheorem{theorem*}{Theorem}
\newtheorem{lemma}{Lemma}
\newtheorem{proposition}[theorem]{Proposition}
\newtheorem{definition}{Definition}

\newtheorem{remark}{Remark}

\newcommand{\RR}{\mathbb{R}}

\newcommand{\NN}{\mathbb{N}}

\newcommand{\HH}{\mathbb{H}}

\newcommand{\Om}{\Omega}

\numberwithin{equation}{section}
\numberwithin{theorem}{section}
\numberwithin{lemma}{section}
\numberwithin{example}{section}
\allowdisplaybreaks

\makeatletter
\@namedef{subjclassname@2020}{\textup{2020} Mathematics Subject Classification}
\makeatother

\begin{document}
    
\title[Sharp fractional Hardy's inequality for half-spaces in $\HH^n$]{Sharp fractional Hardy's inequality for half-spaces in the Heisenberg group}


    \author{Haripada Roy}
    \email{haripada@iitk.ac.in}

    \address{Indian Institute of Technology Kanpur, India}

    \date{}

    \smallskip
    \begin{abstract}
    In this work we establish the following fractional Hardy’s inequality
    $$C\int_{\mathbb{H}^n_+}\frac{|f(\xi)|^p}{x_1^{sp+\alpha}}d\xi\leq \int_{\mathbb{H}^n}\int_{\mathbb{H}^n}\frac{|f(\xi)-f(\xi')|^p}{d({\xi}^{-1}\circ \xi')^{Q+sp}|z'-z|^\alpha}d\xi'd\xi,\ \ \forall\,f\in C_c^\infty(\mathbb{H}^n_+)$$
    for the half-space $\mathbb{H}^n_+=\{\xi=(x,y,t)=(x_1,\ldots,x_n,y_1,\ldots,y_n)\in\mathbb{H}^n:x_1>0\}$ in the Heisenberg group $\mathbb{H}^n$ without any restriction on parameters, and compute the corresponding sharp constant. In a previous joint work, we established a variant of Hardy’s inequality for the same half-space, but with certain parameter restrictions. However, all integrals in that work were considered over half-spaces, and here the seminorm is taken over the entire $\mathbb{H}^n$. Although this inequality holds for all values of the quantity $sp+\alpha$, we are only able to compute the corresponding sharp constant when $sp+\alpha>1$.
    \end{abstract}
    
    \maketitle


\section{Introduction}
Hardy's inequality is a fundamental and extensively studied functional inequality in the theory of partial differential equations. For the Euclidean space $\RR^n$ ($n\geq2$), this inequality takes the form
\begin{equation}
\left|\frac{n-p}{p}\right|^p\int_{\RR^n}\frac{|u(x)|^p}{|x|^p}dx\leq \int_{\RR^n}|\nabla u(x)|^pdx,
\end{equation}
which holds for all $u\in C_c^\infty(\RR^n)$ when $1<p<n$ and for all $u\in C_c^\infty(\RR^n\setminus\{0\})$ when $p>n$. The constant $|\frac{n-p}{p}|^p$ is sharp here. Let $\Omega$ be a domain in $\RR^n$, $n\geq1$, with nonempty boundary, and let us define the distance function from the boundary of $\Omega$ by $$\delta_\Omega(x)=\mathrm{dist}(x,\partial\Omega)=\inf\{|x-y|:y\in\partial\Omega\}.$$ The following variant of Hardy's inequality was established by Alano Ancona in \cite{Ancona} for a sufficiently regular domain $\Omega$:
\begin{equation}\label{bdry har}
C_\Omega\int_\Omega\frac{|u(x)|^2}{\delta_\Omega(x)^2}dx\leq \int_\Omega|\nabla u(x)|^2dx\ \ \forall\,u\in C_c^\infty(\Omega).
\end{equation}
Later, this inequality was generalized by several authors for various domains, such as bounded convex domains, bounded domains with smooth or Lipschitz boundaries, upper half-spaces, and more (see \cite{Hoffmann, Tanya, Tidblom, Tidblom2}). Some of the corresponding sharp constants were also established there.\smallskip

Analogous versions of these inequalities for fractional Sobolev spaces have also been studied for several domains. We refer \cite{Bourgain,Hitchhiker,Frank,Mazya} for introduction of the fractional Sobolev spaces, its connection with the classical spaces, and for the Hardy's inequality for fractional Sobolev spaces. Among the fractional Hardy's inequalities for domains, one of the well-known results was established by B. Dyda in \cite{Dyda}. For $p>0$ and $0<s<1$, he established the following inequality:
\begin{equation}\label{eq Dyda}
C\int_\Omega\frac{|u(x)|^p}{\delta_\Omega(x)^{sp}}dx\leq \int_\Omega\int_\Omega\frac{|u(x)-u(y)|^p}{|x-y|^{n+sp}}dydx\ \ \forall\,u\in C_c^\infty(\Omega),
\end{equation}
where $\Omega$ is either a bounded Lipschitz domain or its complement, a domain above the graph of a Lipschitz function $\RR^{n-1}\rightarrow\RR$, or the complement of a point, with certain conditions imposed on the quantity $sp$.\smallskip

This inequality for the critical case $sp=1$ is addressed in \cite{Purbita_Adi_Pro} for dimension one and in \cite{Adi_Vivek_Pro} for higher dimensions. Inequality \eqref{eq Dyda} for the upper half-space $\RR^n_+=\{x\in\RR^n:x_n>0\}$ and the corresponding sharp constant
\begin{equation}\label{best const}
C_{n,p,s}:=2\pi^{\frac{n-1}{2}}\frac{\Gamma((1+sp)/2)}{\Gamma((n+sp)/2)}\int_0^1\left|1-r^{(sp-1)/p}\right|^p\frac{dr}{(1-r)^{1+sp}}
\end{equation}
was first established by Bogdan and Dyda in \cite{Dyda2} for the case $p=2$ under the condition $sp\neq1$. The results for general $p$ were later extended by Frank and Seiringer in \cite{Frank2}. L. Brasco and E. Cinti established the following variant of the fractional Hardy's inequality for convex domains $(\neq\RR^n)$ in \cite{BraCin} without any restriction on the quantity $sp$:
\begin{equation}\label{brasco hardy}
\frac{C}{s(1-s)}\int_\Omega\frac{|u(x)|^p}{\delta_\Omega(x)^{sp}}dx\leq \int_{\RR^n}\int_{\RR^n}\frac{|u(x)-u(y)|^p}{|x-y|^{n+sp}}dydx\ \ \forall\,u\in C_c^\infty(\Omega),
\end{equation}
where $1<p<\infty$, $0<s<1$ and $C=C(n,p)>0$. Later, Brasco et al. in \cite{Brasco2} established the corresponding sharp constant of the above inequality for the following cases:\\
(1) $\Omega$ is the half-space $\RR^n_+:=\{(x_1,\ldots,x_n)\in\RR^n:x_n>0\}$, and $0<s<1$.\\
(2) $\Omega\subsetneq\RR^n$ is an open convex set and $1/p<s<1$.\\
(3) $\Omega\subsetneq\RR^n$ is an open convex set, $p=2$ and $0<s<1$.\\
The constant is same for all the above cases, which is
\begin{equation}\label{sharp c}
\mathfrak{h}_{s,p}(\Omega)=C_{n,s,p}\Lambda_{s,p},
\end{equation}
where
\begin{equation}
C_{n,s,p}=
\begin{cases}
|S_{n-1}|\int_0^\infty r^{n-2}(1+r^2)^{-\frac{n+sp}{2}}dr, &\text{for $n\geq2$},\\
1 &\text{for $n=1$},
\end{cases}
\end{equation}
$S_{n-1}$ is the unit sphere in $\RR^n$, and
\begin{equation}\label{Lambda}
\Lambda_{s,p}=2\int_0^1\frac{\left|1-t^{\frac{sp-1}{p}}\right|^p}{(1-t)^{1+sp}}dt+\frac{2}{sp}.
\end{equation}\smallskip

The goal of this paper is to establish a Hardy's inequality similar to \eqref{brasco hardy} and the corresponding sharp constant for half-spaces in the Heisenberg group. Before going to the main results, let us recall needed background on the Heisenberg group. The Heisenberg group $\HH^n$, $n\geq1$, is defined by
$$\HH^n:=\{\xi=(z,t)=(x,y,t):z=(x,y)\in \RR^n\times \RR^n\ \mathrm{and}\ t\in \RR\},$$
with the following group law:
$$\xi\circ\xi'=(x+x',y+y',t+t'+2\langle y,x'\rangle-2\langle x,y'\rangle),$$
where $\xi=(x,y,t)$, $\xi'=(x',y',t')\in \HH^n$ and $\langle\cdot,\cdot\rangle$ denotes the usual Euclidean inner product in $\RR^n$. One can easily check that $0\in\HH^n$ is the identity element and $-\xi=(-z,-t)$ is the inverse of $\xi\in\HH^n$. A basis for the left invariant vector fields is given by
$$X_i=\frac{\partial}{\partial x_i}+2y_i\frac{\partial}{\partial t},\ \ 1\leq i\leq n,$$
$$Y_i=\frac{\partial}{\partial y_i}-2x_i\frac{\partial}{\partial t},\ \ 1\leq i\leq n,$$
$$T=\frac{\partial}{\partial t}.$$
For any function $u$ of $C^1$ class, $\nabla_{\HH^n}(u):=\left(X_1(u),\ldots,X_n(u),Y_1(u),\ldots,Y_n(u)\right)$ is the associated Heisenberg sub-gradient of $u$ and
$$|\nabla_{\HH^n}(u)|^2=\sum_{i=1}^{n}\left(|X_i(u)|^2+|Y_i(u)|^2\right).$$
The homogeneous norm for $\xi\in\HH^n$ is defined by
$$d(\xi)=\left((|x|^2+|y|^2)^2+t^2\right)^{\frac{1}{4}}=\left(|z|^4+t^2\right)^{\frac{1}{4}}.$$
$d(\xi)$ is called the Koranyi-Folland non-isotropic gauge. The left-invariant Haar measure on $\HH^n$ is the Lebesgue measure on $\RR^{2n+1}$. For any measurable set $A\subset\HH^n$, $|A|$ denotes the Lebesgue measure of $A$. The dilation on $\HH^n$ is defined by
$$\delta_r\xi:=(rx,ry,r^2t),\ \ \mathrm{for}\ \xi=(x,y,t)\in \mathbb{H}^n\ \mathrm{and}\ r>0.$$
The determinant of the Jacobian matrix of $\delta_r$
is $r^Q$, where $Q=2n+2,$ is called the homogeneous dimension of $\mathbb{H}^n$. A ball in $\HH^n$, centered at $\xi$ with radius $r$ is defined by
\begin{equation}\label{set B_r}
B(\xi,r)=\{\xi'=(z',t'):d({\xi}^{-1}\circ \xi')<r\}.
\end{equation}
For $0<r<R$, the following set is an infinite strip between $x_1=r$ and $x_1=R$: 
\begin{equation}\label{set D_r,R}
\mathscr{D}_{r,R}=\{\xi=(x_1,\ldots,x_n,y,t)\in\HH^n:r<x_1<R\}.
\end{equation}\smallskip

The distance function for a domain $\Omega\subset\HH^n$ with non-empty boundary is defined by
$$\delta_\Omega(\xi)=\begin{cases}
\underset{\xi'\in\partial\Omega}{\mathrm{inf}}d(\xi^{-1}\circ\xi') & \text{if $\xi\in\Omega$}\\
0 & \text{if $\xi\in\HH^n\setminus\Omega$}.
\end{cases}$$
Let us define the following half-spaces in $\HH^n$:
$$\HH^n_{+}:=\{\xi=(x,y,t)=(x_1,\ldots,x_n, y_1,\ldots,y_n,t)\in\HH^n:x_1>0\}$$
and
$$\HH^n_{+,t}:=\{\xi=(x,y,t)\in\HH^n:t>0\}.$$
Considering the point $\xi'=(0,x_2,\ldots,x_n,y,t-2x_1y_1)\in\partial\HH^n_+$, it is easy to see that $\delta_{\HH^n_{+}}(\xi)=x_1$ for $\xi\in\HH^n_+$, whereas the distance function $\delta_{\HH^n_{+,t}} (\xi) \leq \sqrt{t}$ for the half-space $\HH^n_{+,t}$. In fact, the distance of the point $\xi=(0,0,t)\in \HH^n_{+,t}$ from the boundary of $\HH^n_{+,t}$ is $\sqrt{t}$, but the value could be strictly less than $\sqrt{t}$ for $\xi=(x,y,t)$ in some subsets of $\HH^n_{+,t}$, see the Appendix in \cite{Hari3}. Moreover, there is no known specific form for the distance function $\delta_{\HH^n_{+,t}}(\xi)$.\smallskip

Hardy's inequality for the Heisenberg group was first established by Garofalo and Lanconelli in \cite{Garofalo}. Numerous generalizations of this inequality have been made by various authors. The following inequality is one of the general version, proved by L. D'Ambrozio in \cite{D'Ambrosio}:
\begin{equation}\label{eq D'Ambro}
\int_{\Omega}\frac{|z|^\beta}{d(\xi)^\alpha}|u(\xi)|^pd\xi\leq C_{n,p,\alpha,\beta}^p\int_{\Omega}|\nabla_{\HH^n}(u)|^p |z|^{\beta-p}d(\xi)^{2p-\alpha},\ \forall u\in C_c^{\infty}(\Omega),
\end{equation}
where $\Omega$ is an open subset of $\HH^n$, $n\geq1$, $p>1$, $\alpha,\beta\in\RR$ satisfy $Q>\alpha-\beta$ and $2n>p-\beta$ and $C_{n,p,\alpha\beta}=\frac{p}{Q+\beta-\alpha}$ is the optimal constant.\smallskip

Inequality \eqref{eq D'Ambro} was proved for $\Omega=\HH^n$ with $\alpha=2p$ and $\beta=p$ in \cite{Pengcheng3}. Adimurthi and Sekar proved \eqref{eq D'Ambro} with $\alpha=2p$ and $\beta=2$, for $u\in FS_0^{1,p}(\HH^n)$ in \cite{Adi_Sekar_2006}, where the space $FS_0^{1,p}(\HH^n)$ is the completion of $C_c^{\infty}(\HH^n)$ under the norm
\begin{equation}\label{FS norm}
|u|^p_{1,p}=\int_{\HH^n}\frac{|\nabla_{\HH^n}(u)|^p}{|z|^{p-2}}dzdt.
\end{equation}
For other generalizations of this inequality, we refer to the works in \cite{Pengcheng,Larson,Luan,Xiao}.\bigskip

Let $1\leq p<\infty$, $0<s<1$, $\alpha\geq0$, and $\Omega\subset\HH^n$ be an open set. Let us define the following fractional seminorm analogous to \eqref{FS norm}: 
\begin{equation}\label{seminorm}
[f]_{s,p,\alpha,\Omega}:=\left(\int_{\Omega}\int_{\Omega}\frac{|f(\xi)-f(\xi')|^p}{d({\xi}^{-1}\circ \xi')^{Q+sp}|z-z'|^{\alpha}}d\xi' d\xi\right)^{\frac{1}{p}},
\end{equation}
and the fractional Sobolev space
$$W^{s,p,\alpha}(\Omega):=\left\{f\in L^p(\Omega):[f]_{s,p,\alpha,\Omega}<\infty\right\}.$$
Here the norm is defined by $\lVert f\rVert_{s,p,\alpha,\Omega}^p=\lVert f\rVert_{L^p(\Omega)}^p+[f]_{s,p,\alpha,\Omega}^p$. The space $W_0^{s,p,\alpha}(\Omega)$ is defined as the closure of $C_c^\infty(\Omega)$ with respect to the norm $\lVert f\rVert_{s,p,\alpha,\Omega}$. When $\alpha=0$, the above notations are simplified to $[f]_{s,p,\Omega}$, $\lVert f\rVert_{s,p,\Omega}$, $W^{s,p}(\Omega)$, and $W_0^{s,p}(\Omega)$, respectively. Hardy's inequality for the space $W_0^{s,p,\alpha}(\Omega)$ takes the form
\begin{equation}\label{eq adi}
C_{n,s,p,\alpha} \int_\Omega\frac{|f(\xi)|^p}{\delta_{\Omega}(\xi)^{sp}|z|^{\alpha}}d\xi\leq [f]_{s,p,\alpha,\Omega}^p.
\end{equation}
When $\Om=\HH^n$, the distance function $\delta_{\Omega}(\xi)$ is replaced by $d(\xi)$. The space $W_0^{s,p,\alpha}(\HH^n)$ was first introduced in \cite{Adi-Mallick}. This space is trivial when $\alpha\geq Q-2$, and if $\alpha<\min\{Q-2,p(1-s)\}$, then $C_c^\infty(\HH^n)\subset W_0^{s,p,\alpha}(\HH^n)$ (see Section 3 in \cite{Adi-Mallick}). Thus, throughout this paper, we consider $\alpha<Q-2$. Inequality \eqref{eq adi} for this space was established in \cite{Adi-Mallick} under the condition $sp+\alpha<Q$ together with $sp>2$. Recently, the range of the indices for this inequality has been improved as a particular case of fractional Caffarelli-Kohn-Nirenberg type inequalities in \cite{Hari}.\smallskip

In a previous joint work \cite{Hari3}, we established the fractional Hardy inequality \eqref{eq adi} for the half-space $\HH^n_+$ under conditions $sp>1$ and $\alpha\geq (Q-2+sp)/2$. In this paper, following the techniques in \cite{BraCin} and \cite{Brasco2}, we establish a fractional Hardy's inequality analogous to \eqref{brasco hardy}, for the half-space $\HH^n_+$, without any restriction on the product $sp$, along with the corresponding sharp constant. The inequality is as follows:

\begin{theorem}\label{thm 1.1}
Let $1<p<\infty$ and $0<s<1$. There exists a constant $C=C(n,p)>0$ depending only on $n$ and $p$, such that
\begin{equation}\label{eq main}
\frac{C}{s(1-s)}\int_{\HH^n_+}\frac{|f(\xi)|^p}{x_1^{sp}}d\xi\leq \int_{\HH^n}\int_{\HH^n}\frac{|f(\xi)-f(\xi')|^p}{d({\xi}^{-1}\circ \xi')^{Q+sp}}d\xi' d\xi\ \ \forall\,f\in C_c^\infty(\HH^n_+),
\end{equation}
where $\xi=(x_1,\cdots x_n,y_1,\cdots,y_n,t)$.
\end{theorem}\smallskip

The right-hand side of the above inequality is the seminorm \eqref{seminorm} without the index $\alpha$. Since we are interested in establishing an inequality for the space $W^{s,p,\alpha}$, the following version of Hardy's inequality precisely has the seminorm \eqref{seminorm} on the right-hand side:
\begin{theorem}\label{thm alpha>0}
Let $1<p<\infty$, $0<s<1$, and $\alpha\geq0$. There exists a constant $C=C(n,p,\alpha)>0$ depending only on $n$, $p$ and $\alpha$, such that
\begin{equation}\label{eq alpha>0}
\frac{C}{(s+\frac{\alpha}{p})(1-s)}\int_{\HH^n_+}\frac{|f(\xi)|^p}{x_1^{sp+\alpha}}d\xi\leq [f]_{s,p,\alpha,\HH^n}^p.\ \ \forall\,f\in C_c^\infty(\HH^n_+),
\end{equation}
where $\xi=(z,t)=(x_1,\cdots x_n,y_1,\cdots,y_n,t)$.
\end{theorem}
Inequality \eqref{eq adi} for the half-space $\HH^n_+$ with $0\leq\alpha<(Q-2+sp)/2$ has not yet been established. However, observing that $x_1^\alpha\leq|z|^\alpha$, inequality \eqref{eq alpha>0} implies that \eqref{eq adi} holds for every $\alpha\geq0$, provided that the right-hand side has $[f]_{s,p,\alpha,\HH^n}$ instead of $[f]_{s,p,\alpha,\Omega}$. Although Theorem \ref{thm alpha>0} is more general than Theorem \ref{thm 1.1}, (Theorem \ref{thm 1.1} is the special case $\alpha=0$ of Theorem \ref{thm alpha>0}), its proof is based on Theorem \ref{thm 1.1}.\bigskip

The constants given in \eqref{eq main} and \eqref{eq alpha>0} may not be sharp. It is enough to compute the sharp constant for \eqref{eq alpha>0} only. The sharp constant for the inequality \eqref{eq alpha>0} is defined by the variational problem
\begin{equation}\label{sharp C}
\mathcal{C}_{s,p,\alpha}(\HH^n_+)=\underset{f\in C_c^\infty(\HH^n_+)}{\inf}\left\{[f]_{s,p,\alpha,\HH^n}^p:\int_{\HH^n_+}\frac{|f(\xi)|^p}{x_1^{sp+\alpha}}d\xi=1\right\}.
\end{equation}
The existence of such a positive constant follows from Theorem \ref{thm alpha>0}. In fact, we have
$$\frac{C(n,p,\alpha)}{(s+\frac{\alpha}{p})(1-s)}\leq \mathcal{C}_{s,p,\alpha}(\HH^n_+).$$
Before defining the precise form of the sharp constant for inequality \eqref{eq alpha>0}, let us introduce some notation. For $1<p<\infty$, $0<s<1$ and $\alpha\geq0$, define
\begin{equation}\label{Lambda sp}
\Lambda_{s,p,\alpha}:=2\int_0^1\frac{\left|1-\tau^{\frac{sp+\alpha-1}{p}}\right|^p}{(1-\tau)^{1+sp+\alpha}}d\tau+\frac{2}{sp+\alpha}.
\end{equation}
Note that $\Lambda_{s,p,\alpha}$ is the same as $\Lambda_{s',p}$ as defined in \eqref{Lambda} with $s'=s+\frac{\alpha}{p}<1$. For $k\in\NN$ with $k\geq2$, $\theta>0$ and $\alpha\geq0$, let us define the quantity
\begin{equation}\label{eq I}
\mathcal{I}(k,\theta,\alpha)=\int_0^\infty\frac{dt}{(1+t^2)^{\frac{k+2+\theta}{4}}}\int_0^\infty\frac{r^{k-2}dr}{(1+r^2)^{\frac{k+\theta+\alpha}{2}}}.
\end{equation}
\begin{theorem}\label{thm sharp C}
Let $1<p<\infty$, $0<s<1$ and $\alpha<\min\{Q-2,p(1-s)\}$ be such that $sp+\alpha>1$. Then we have
\begin{equation}
\mathcal{C}_{s,p,\alpha}(\HH^n_+)=C_{n,s,p,\alpha}\Lambda_{sp},
\end{equation}
where
\begin{equation}\label{eq C(n,s,p)}
C_{n,s,p,\alpha}=2|S_{2n-2}|\mathcal{I}(Q-2,sp,\alpha),
\end{equation}
$S_{2n-2}$ is the unit sphere in $\RR^{2n-1}$.
\end{theorem}

In the sharp constant result, there are two restrictions on the parameters: the first is $\alpha<\min\{Q-2,p(1-s)\}$, and the second is $sp+\alpha>1$. The following integral
$$\int_{|z|<1}\int_{|t|<1}\frac{dzdt}{(|z|^4+t^2)^{\frac{Q+sp}{4}}|z|^\alpha}$$ 
is required to be finite in the proof. This is true provided $\alpha<\min\{Q-2,p(1-s)\}$ (see integral $J$ in Proposition 3.1 in \cite{Adi-Mallick}).\smallskip

Although we follow the technique used in \cite{Brasco2} for the half-space $\RR^n_+=\{x\in\RR^n:x_n>0\}$ to compute the sharp constant for inequality \eqref{eq alpha>0}, the proof for $\HH^n_+$ differs significantly. In \cite{Brasco2} Brasco et al. use a test function of the form $\phi(x)=\chi_M(x')\eta(x_n)$, $M>0$, to separate the integrals in $\RR^{n-1}$ and $\RR$, and use the one-dimensional result. In our case, it is not possible to separate the term $d({\xi}^{-1}\circ \xi')^{Q+sp}|z-z'|^{\alpha}$ in the denominator (even in the case $\alpha=0$) in terms of $x_1$ and $(\tilde{x},t)$, (here $\tilde{x}=(x_2,\ldots,x_n,y_1,\ldots,y_n)$ and $\xi=(x_1,\tilde{x},t)$). To overcome this issue, we use a test function of the form $\phi(\xi)=\eta(x_1)\Psi_{R_0}(\xi)$, $R_0>0$, along with a certain monotonicity property of $\Psi_{R_0}$ in the variable $x_1$. Additionally, the condition $sp+\alpha>1$ is required here. To deal with the term $|z-z'|^\alpha$, $\alpha>0$, we also introduce some different spaces of functions that are integrable over the infinite strip $\mathscr{D}_r,R$, which differs this work from Euclidean results.\bigskip

The paper is organized as follows: In Section \ref{sec 2}, we introduce a weighted fractional Laplacian operator related to the seminorm and fractional Sobolev spaces considered in this paper, some associated function spaces, the notion of weakly superharmonic and subharmonic functions with respect to this operator, and several necessary results regarding the integrability of such functions. Section \ref{sec 3} contains some useful properties of the distance function of $\HH^n$. In Section \ref{sec 4}, we provide the proofs of the fractional Hardy's inequalities Theorem \ref{thm 1.1} and Theorem \ref{thm alpha>0}. The computation of the sharp constants of these inequalities is given in Section \ref{sec 5}.

\section{Preliminaries}\label{sec 2}

For $1<p<\infty$, $0<s<1$ and $\alpha\geq0$ let us define the operator
\begin{equation}
\mathcal{L}f(\xi):=\mathrm{P.V.}\int_{\HH^n}\frac{J_p(f(\xi)-f(\xi'))}{d({\xi}^{-1}\circ \xi')^{Q+sp}|z-z'|^{\alpha}}d\xi',\ \ \xi\in\HH^n,
\end{equation}
where $J_p:\RR\rightarrow\RR$ is a monotonically increasing continuous function defined by
$$J_p(t)=|t|^{p-2}t\ \ \mathrm{for}\ t\in\RR.$$
It is easy to see that the weak formulation of the operator $\mathcal{L}$ precisely gives the seminorm defined in \eqref{seminorm}. We refer to \cite{Piccinini} for more details on this operator. To deal with the integrals involving $|z-z'|^\alpha$, let us introduce the following spaces:
$$L_{\mathrm{loc}(x_1)}^{p}(\HH^n_+):=\left\{f\ \mathrm{is}\ \mathrm{measurable}:\int_{\mathscr{D}_{r,R}}|f(\xi)|^pd\xi<\infty\ \ \mathrm{for}\ \mathrm{any}\ 0<r<R\right\},$$
and
$$W_{\mathrm{loc}(x_1)}^{s,p,\alpha}(\HH^n_+):=\left\{f\in L_{\mathrm{loc}(x_1)}^{p}(\HH^n_+):[f]_{{s,p,\alpha,\mathscr{D}_{r,R}}}<\infty\ \ \mathrm{for}\ \mathrm{any}\ 0<r<R\right\}.$$
Clearly, $W^{s,p,\alpha}(\HH^n_+)\subset W_{\mathrm{loc}(x_1)}^{s,p,\alpha}(\HH^n_+)\subset W_{\mathrm{loc}}^{s,p,\alpha}(\HH^n_+)$. Let us also consider the space
$$L_{sp,\alpha}^{p-1}(\HH^n):=\left\{f\in L_{\mathrm{loc}}^{p-1}(\HH^n):\int_{\HH^n}\frac{|f(\xi)|^{p-1}}{(1+d(\xi))^{Q+sp}(1+|z|)^\alpha}d\xi<\infty\right\}.$$
One can easily check that $L^\infty(\HH^n)\subset L_{sp,\alpha}^{p-1}(\HH^n)$. Let $\Omega\subseteq\HH^n_+$ be any open set ($\Omega$ could be $\HH^n_+$ itself). By extending the analogous idea from the Euclidean case in \cite{Brasco2} to the Heisenberg group, let us consider the following equation:
\begin{equation}\label{eq L}
\mathcal{L}f(\xi)=\lambda\frac{J_p(f(\xi))}{x_1^{sp+\alpha}}\ \ \mathrm{in}\ \Omega.
\end{equation}
\begin{definition}\label{def superharmonic}
\emph{(i)} A function $f\in W_{\mathrm{loc}(x_1)}^{s,p,\alpha}(\HH^n)\cap L_{sp,\alpha}^{p-1}(\HH^n)$ is said to be a weak supersolution (or subsolution, respectively) to \eqref{eq L} if
\begin{equation}\label{eq super-sub}
\int_{\HH^n}\int_{\HH^n}\frac{J_p(f(\xi)-f(\xi'))(\phi(\xi)-\phi(\xi'))}{d({\xi}^{-1}\circ \xi')^{Q+sp}|z-z'|^\alpha}d\xi' d\xi\geq(\leq,\ \mathrm{resp})\,\lambda\int_{\HH^n_+}\frac{J_p(f(\xi))\phi(\xi)}{x_1^{sp+\alpha}}d\xi,
\end{equation}
for every non-negative $\phi\in W^{s,p,\alpha}(\HH^n)$ with compact support in $\Omega$. A function $f$ is called a weak solution to \eqref{eq L}, if it is both a weak supersolution and a weak subsolution to \eqref{eq L}.\smallskip

\noindent \emph{(ii)} A function $f\in W_{\mathrm{loc}(x_1)}^{s,p,\alpha}(\HH^n_+)\cap L_{sp,\alpha}^{p-1}(\HH^n)$ is called weakly superharmonic, weakly subharmonic, or weakly harmonic if the integral on the left-hand side of \eqref{eq super-sub} is greater than or equal to zero, less than or equal to zero, or equal to zero, respectively.
\end{definition}
The following lemma shows that, under the assumptions on $f$ and the test function $\phi$, the above definition is well-posed, i.e., the function
\begin{equation}
(\xi,\xi')\mapsto\frac{J_p(f(\xi)-f(\xi'))}{d({\xi}^{-1}\circ \xi')^{Q+sp}|z-z'|^\alpha}(\phi(\xi)-\phi(\xi'))
\end{equation}
is integrable.
\begin{lemma}\label{Lemma 2.1}
Let $1<p<\infty$, $0<s<1$ and $\alpha\geq0$. If $f\in W_{\mathrm{loc}(x_1)}^{s,p,\alpha}(\HH^n_+)\cap L_{sp,\alpha}^{p-1}(\HH^n)$ and $\phi\in W^{s,p,\alpha}(\HH^n)$ with compact support in $\Omega$, then we have
\begin{equation}
\int_{\HH^n}\int_{\HH^n}\frac{|f(\xi)-f(\xi')|^{p-1}|\phi(\xi)-\phi(\xi')|}{d({\xi}^{-1}\circ \xi')^{Q+sp}|z-z'|^\alpha}d\xi' d\xi<\infty.
\end{equation}
\end{lemma}
\begin{proof}
Let $\mathcal{O}$ be the support of $\phi$. Since $\mathcal{O}$ is a compact subset of $\Omega$ ($\subset\HH^n_+$), $\exists\,r,R$, such that
$$r:=\inf\{x_1:\xi=(x_1,\ldots,x_n,y,t)\in\mathcal{O}\}>0\ \ \mathrm{and}\ \ R:=\sup\{x_1:\xi=(x_1,\ldots,x_n,y,t)\in\mathcal{O}\}<\infty.$$
Clearly $\mathcal{O}\Subset\mathscr{D}_{\frac{r}{2},R+\frac{r}{2}}\subset\HH^n_+$. For simplicity of notation, the set $\mathscr{D}_{\frac{r}{2},R+\frac{r}{2}}$ is denoted by $\mathscr{D}$ here. If $\xi\in \mathcal{O}$ and $\xi'\in\mathscr{D}^C$, then we have $d({\xi}^{-1}\circ \xi')\geq |x_1-x_1'|\geq\frac{r}{2}$.\\
Now, the above integral can be decomposed as follows:
\begin{multline*}
\int_{\HH^n}\int_{\HH^n}\frac{|f(\xi)-f(\xi')|^{p-1}|\phi(\xi)-\phi(\xi')|}{d({\xi}^{-1}\circ \xi')^{Q+sp}|z-z'|^\alpha}d\xi' d\xi=\int_{\mathscr{D}}\int_{\mathscr{D}}\frac{|f(\xi)-f(\xi')|^{p-1}|\phi(\xi)-\phi(\xi')|}{d({\xi}^{-1}\circ \xi')^{Q+sp}|z-z'|^\alpha}d\xi' d\xi\\
+2\int_{\mathcal{O}}\int_{\mathscr{D}^C}\frac{|f(\xi)-f(\xi')|^{p-1}|\phi(\xi)|}{d({\xi}^{-1}\circ \xi')^{Q+sp}|z-z'|^\alpha}d\xi' d\xi=:I_1+2I_2.
\end{multline*}
The integral $I_1$ can be shown to be finite using H\"older's inequality with the conjugates $\frac{p}{p-1}$ and $p$, and with the assumptions on $f$ and $\phi$. For the second integral, we have
\begin{multline*}
I_2\leq C \int_{\mathcal{O}}|f(\xi)|^{p-1}|\phi(\xi)|\int_{\mathscr{D}^C}\frac{d\xi'}{d({\xi}^{-1}\circ \xi')^{Q+sp}|z-z'|^\alpha}d\xi\\
+C\int_{\mathcal{O}}|\phi(\xi)|\int_{\mathscr{D}^C}\frac{|f(\xi')|^{p-1}d\xi'}{d({\xi}^{-1}\circ \xi')^{Q+sp}|z-z'|^\alpha}d\xi=:C(I_3+I_4).
\end{multline*}
For any $\xi\in\mathcal{O}$, we have
$$\int_{\mathscr{D}^C}\frac{d\xi'}{d({\xi}^{-1}\circ \xi')^{Q+sp}|z-z'|^\alpha}\leq\int_{B(\xi,\frac{r}{2
})^C}\frac{d\xi'}{d({\xi}^{-1}\circ \xi')^{Q+sp}|z-z'|^\alpha}=\int_{B(0,\frac{r}{2
})^C}\frac{d\xi'}{d(\xi')^{Q+sp}|z'|^\alpha}.$$
The last integral was shown to be finite in \cite{Adi-Mallick} (see the integral $I_3$ in Proposition 3.1.). Therefore, we get
$$I_3\leq C_0\int_{\mathcal{O}}|f(\xi)|^{p-1}|\phi(\xi)|d\xi\leq C_0\lVert f\rVert_{L^p(\mathcal{O})}^{p-1}\lVert \phi\rVert_{L^p(\mathcal{O})}<\infty.$$
Observe that $\frac{1+d(\xi)}{d({\xi}^{-1}\circ \xi')}\leq C_1$ and $\frac{1+|z|}{|z-z'|}\leq C_2$ for any $\xi\in\mathcal{O}$ and $\xi'\in\mathscr{D}^C$. Therefore, we have
$$I_4\leq C C_1^{Q+sp}C_2^\alpha\left(\int_{\mathcal{O}}|\phi(\xi)|d\xi\right)\left(\int_{\mathscr{D}^C}\frac{|f(\xi')|^{p-1}d\xi'}{(1+d(\xi))^{Q+sp}(1+|z|)^\alpha}\right)<\infty,$$
This completes the proof.
\end{proof}
\begin{remark}
\emph{For the case $\alpha=0$, it is enough to consider $f\in W_{\mathrm{loc}}^{s,p,0}(\HH^n_+)$ in Definition \ref{def superharmonic}, and in this case the space $L_{sp,0}^{p-1}(\HH^n)$ is denoted by $L_{sp}^{p-1}(\HH^n)$.}
\end{remark}
The next lemma states that the product of two certain functions belongs to the space $W^{s,p}(\Omega)$ and ensures the admissibility of some test functions in our main result.

\begin{lemma}\label{lemma 3.2}
Let $1<p<\infty$, $0<s<1$, $\alpha\geq0$, and $\Omega\subseteq\HH^n_+$ be an open set. If $\phi\in W^{s,p,\alpha}(\Omega)\cap L^\infty(\Omega)$ with compact support in $\Omega$ and $f\in W^{s,p,\alpha}_{\mathrm{loc}(x_1)}(\HH^n_+)\cap L^\infty(\HH^n_+)$, then
$$f\phi\in W^{s,p,\alpha}(\Omega).$$
\end{lemma}
\begin{proof}
We have
\begin{multline}\label{eq f phi}
[f\phi]_{s,p,\alpha,\Omega}^p\\
\leq 2^{p-1}\int_{\Omega}\int_{\Omega}\frac{|\phi(\xi)-\phi(\xi')|^p|f(\xi)|^p}{d({\xi}^{-1}\circ \xi')^{Q+sp}|z-z'|^\alpha}d\xi' d\xi+2^{p-1}\int_{\Omega}\int_{\Omega}\frac{|f(\xi)-f(\xi')|^p|\phi(\xi')|^p}{d({\xi}^{-1}\circ \xi')^{Q+sp}|z-z'|^\alpha}d\xi' d\xi\\
\leq 2^{p-1}\lVert f\rVert_{L^\infty(\Omega)}^p[\phi]_{s,p,\alpha,\Omega}^p+2^{p-1}\int_{\Omega}\int_{\Omega}\frac{|f(\xi)-f(\xi')|^p|\phi(\xi')|^p}{d({\xi}^{-1}\circ \xi')^{Q+sp}|z-z'|^\alpha}d\xi' d\xi.
\end{multline}
It remains to prove that the last integral is finite. Let $\mathcal{O}$ be the support of $\phi$. Since $\mathcal{O}$ is a compact subset of $\Omega\subseteq\HH^n_+$, similarly to the proof of Lemma \ref{Lemma 2.1}, the set $\mathscr{D}=\mathscr{D}_{\frac{r}{2},R+\frac{r}{2}}$ can be defined, where $\mathcal{O}\Subset\mathscr{D}\subset\HH^n_+$ and $d({\xi}^{-1}\circ \xi')\geq|z-z'|\geq\frac{r}{2}$ whenever $\xi\in \mathcal{O}$ and $\xi'\in{\mathscr{D}}^C$. Using this we obtain
\begin{multline*}
\int_{\Omega}\int_{\Omega}\frac{|f(\xi)-f(\xi')|^p|\phi(\xi')|^p}{d({\xi}^{-1}\circ \xi')^{Q+sp}|z-z'|^\alpha}d\xi' d\xi\\
\leq\int_{\mathscr{D}}\int_{\mathscr{D}}\frac{|f(\xi)-f(\xi')|^p|\phi(\xi')|^p}{d({\xi}^{-1}\circ \xi')^{Q+sp}|z-z'|^\alpha}d\xi' d\xi+2\int_{\mathcal{O}}\int_{\mathscr{D}^C}\frac{|f(\xi)-f(\xi')|^p|\phi(\xi')|^p}{d({\xi}^{-1}\circ \xi')^{Q+sp}|z-z'|^\alpha}d\xi' d\xi\\
\leq \lVert \phi\rVert_{L^\infty(\HH^n_+)}^p\left([f]_{s,p,\alpha,\mathscr{D}}^p+2\int_{\mathcal{O}}\int_{\mathscr{D}^C}\frac{|f(\xi)-f(\xi')|^p}{d({\xi}^{-1}\circ \xi')^{Q+sp}|z-z'|^\alpha}d\xi' d\xi\right)\\
\leq \lVert \phi\rVert_{L^\infty(\HH^n_+)}^p\left([f]_{s,p,\alpha,\mathscr{D}}^p+2^{p+1}\lVert f\rVert_{L^\infty(\HH^n_+)}^p|\mathcal{O}|\int_{B(0,\frac{r}{2})^C}\frac{d\xi'}{d(\xi')^{Q+sp}|z'|^\alpha}\right)<\infty.
\end{multline*}
\end{proof}
\begin{remark}\label{remark 2}
\emph{For the case $\alpha=0$, instead of the space $W^{s,p,\alpha}_{\mathrm{loc}(x_1)}(\HH^n_+)$, we need to work with functions $f\in W^{s,p}_{\mathrm{loc}}(\HH^n_+)$. In this case, to show that the last integral on the right-hand side of \eqref{eq f phi} is finite, define a compact set $\mathcal{O}'$ with $\mathcal{O}\Subset\mathcal{O}'\Subset\HH^n_+$ in the following way:
\begin{equation}\label{set O'}
\mathcal{O}':=\left\{\xi=(x_1,\ldots,x_n,y,t)=(z,t):\frac{r}{2}\leq x_1\leq R+\frac{r}{2},\,|z|\leq R_0+\frac{r}{2}\ \mathrm{and}\ |t|\leq T_0+\frac{rR_0}{2}+\frac{r^2}{4}\right\},
\end{equation}
where
$r:=\inf\{x_1:\xi=(x_1,\ldots,x_n,y,t)\in\mathcal{O}\}>0$, $R:=\sup\{x_1:\xi=(x_1,\ldots,x_n,y,t)\in\mathcal{O}\}$, $R_0=\sup\{|z|:\xi=(z,t)\in\mathcal{O}\}$ and $T_0=\sup\{|t|:\xi=(z,t)\in\mathcal{O}\}$.}\smallskip

\emph{For $\xi\in \mathcal{O}$ and $\xi'\in{\mathcal{O}'}^C$, we have to possible cases. If $|z-z'|\geq\frac{r}{2}$, then we have $d({\xi}^{-1}\circ \xi')\geq\frac{r}{2}$. If $|z-z'|<\frac{r}{2}$, we must have $|t'-t|\geq\frac{rR_0}{2}+\frac{r^2}{4}$, and
$$|t'-t-2\langle y,x'\rangle+2\langle x,y'\rangle|=|t'-t+2\langle (y,-x),z-z'\rangle|\geq |t'-t|-2|z||z-z'|\geq\frac{r^2}{4}.$$
Thus, we always have $d({\xi}^{-1}\circ \xi')\geq\frac{r}{2}$, whenever $\xi\in \mathcal{O}$ and $\xi'\in{\mathcal{O}'}^C$. Using this, similarly, it can be shown that the last integral on the right-hand side of \eqref{eq f phi} is finite.}
\end{remark}

\section{Some properties of the distance function}\label{sec 3}

This section establishes several useful results of the distance function $\delta_{\HH^n_+}(\xi)={x_1}_+$ (${x_1}_+$ denotes $\max\{x_1,0\}$). For the Euclidean case, Brasco and Cinti in \cite{BraCin} have shown that the $s$-th power of the distance function for a bounded convex set is weakly superharmonic. In the Heisenberg group case, we show that the function $(\delta_{\HH^n_+})^s$ is weakly harmonic, and this is the main ingredient in proving our first result, Theorem \ref{thm 1.1}. The parameter $\alpha$ is not involved in Theorem \ref{thm 1.1}; therefore, in this section, we work with $\alpha=0$.

\begin{lemma}\label{lemma 3.1}
Let $1<p<\infty$, $0<s<1$. If $f\in W^{s,p}_{\mathrm{loc}}(\HH^n_+)\cap L_{sp}^{p-1}(\HH^n)$ and $\phi\in L^p(\HH^n)$ with compact support in $\HH^n_+$, then for any $\epsilon>0$, the function
$$(\xi,\xi')\mapsto\frac{J_p(f(\xi)-f(\xi'))}{d({\xi}^{-1}\circ \xi')^{Q+sp}}\phi(\xi)$$
is integrable on $\mathcal{T}_\epsilon:=\{(\xi,\xi')\in\HH^n\times\HH^n:d({\xi}^{-1}\circ \xi')\geq\epsilon\}$.
\end{lemma}
\begin{proof}
Let $\mathcal{O}$ be the support of $\phi$. We have
\begin{multline*}
\iint_{\mathcal{T}_\epsilon}\left|\frac{J_p(f(\xi)-f(\xi'))}{d({\xi}^{-1}\circ \xi')^{Q+sp}}\phi(\xi)\right|d\xi' d\xi\\
=\iint_{\{(\xi,\xi')\in\mathcal{O}\times\HH^n:d({\xi}^{-1}\circ \xi')\geq\epsilon\}}\frac{|f(\xi)-f(\xi')|^{p-1}}{d({\xi}^{-1}\circ \xi')^{Q+sp}}|\phi(\xi)|d\xi' d\xi\\
\leq \epsilon^{-\frac{Q+sp}{p}}\iint_{\{(\xi,\xi')\in\mathcal{O}\times\mathcal{O}:d({\xi}^{-1}\circ \xi')\geq\epsilon\}}\frac{|f(\xi)-f(\xi')|^{p-1}}{d({\xi}^{-1}\circ \xi')^{\frac{Q+sp}{q}}}|\phi(\xi)|d\xi' d\xi\\
+\iint_{\{(\xi,\xi')\in\mathcal{O}\times(\HH^n\setminus\mathcal{O}):d({\xi}^{-1}\circ \xi')\geq\epsilon\}}\frac{|f(\xi)-f(\xi')|^{p-1}}{d({\xi}^{-1}\circ \xi')^{Q+sp}}|\phi(\xi)|d\xi' d\xi,
\end{multline*}
where $q=\frac{p}{p-1}$. Applying H\"older's inequality to the first integral on the right-hand side, we get
\begin{multline*}
\iint_{\{(\xi,\xi')\in\mathcal{O}\times\mathcal{O}:d({\xi}^{-1}\circ \xi')\geq\epsilon\}}\frac{|f(\xi)-f(\xi')|^{p-1}}{d({\xi}^{-1}\circ \xi')^{\frac{Q+sp}{q}}}|\phi(\xi)|d\xi' d\xi\\
\leq \left(\int_\mathcal{O}\int_\mathcal{O}\frac{|f(\xi)-f(\xi')|^{p}}{d({\xi}^{-1}\circ \xi')^{Q+sp}}d\xi' d\xi\right)^{\frac{1}{q}}\left(|\mathcal{O}|\int_\mathcal{O}|\phi(\xi)|^pd\xi\right)^{\frac{1}{p}}<\infty.
\end{multline*}
The second integral on the right-hand side can be shown to be finite similarly to the integral $I_2$ in Lemma \ref{Lemma 2.1}.
\end{proof}

The following result for the distance function $\delta_{\HH^n_+}$ is crucial in the proof of Theorem \ref{thm 1.1}. It is not known whether the same result holds for the half-space $\HH^n_{+,t}$.
\begin{lemma}\label{lemma 3.3}
Let $1<p<\infty$, $0<s<1$. There exists a constant $C=C(n,p)>0$ such that
\begin{equation}\label{eq 3.2}
\int_{\{\xi'\in\HH^n_+:x_1'\leq x_1\}}\frac{|x_1-x_1'|^p}{d({\xi}^{-1}\circ \xi')^{Q+sp}}d\xi'\geq \frac{C}{1-s}x_1^{p-sp}\ \ \mathrm{for}\ \mathrm{a.e}\ \xi\in\HH^n_+,
\end{equation}
where $\xi=(x_1,\ldots,x_n,y,t)$ and $\xi'=(x_1',\ldots,x_n',y',t')$.
\end{lemma}
\begin{proof}
Let $\xi\in\HH^n_+$. For any $\sigma\in(0,1)$ let us define the following set
$$\Sigma_\sigma(\xi,r):=\{\xi'=(x_1',\ldots,x_n',y',t')\in B(\xi,r):x_1-x_1'>\sigma d({\xi}^{-1}\circ \xi')\}.$$
Clearly, $\Sigma_\sigma(\xi,x_1)\subset\{\xi'\in\HH^n_+:x_1'\leq x_1\}$ and we have
\begin{multline}\label{eq 3.3}
\int_{\{\xi'\in\HH^n_+:x_1'\leq x_1\}}\frac{|x_1-x_1'|^p}{d({\xi}^{-1}\circ \xi')^{Q+sp}}d\xi'\geq \int_{\Sigma_\sigma(\xi,x_1)}\frac{|x_1-x_1'|^p}{d({\xi}^{-1}\circ \xi')^{Q+sp}}d\xi'\\
\geq \sigma^p\int_{\Sigma_\sigma(\xi,x_1)} d({\xi}^{-1}\circ \xi')^{p-sp-Q}d\xi'=\sigma^p\int_{\Sigma_\sigma(0,x_1)} d(\xi')^{p-sp-Q}d\xi'.
\end{multline}
The last integral can be written in polar coordinates for $\HH^n$ as
\begin{equation}\label{eq 3.4}
\int_{\Sigma_\sigma(0,x_1)} d(\xi')^{p-sp-Q}d\xi'=\left(\int_{\{\omega\in S_{\HH^n}:-\omega_{x_1}>\sigma d(\omega)\}}dS_{\HH^n}(\omega)\right)\int_0^{x_1}r^{p-sp-1}dr=\frac{\psi(\sigma)}{p(1-s)}x_1^{p-sp},
\end{equation}
where $S_{\HH^n}=\{\xi\in\HH^n:d(\xi)=1\}$, $\omega=(\omega_{x_1},\ldots,\omega_{x_n},\omega_{y_1},\ldots,\omega_{y_n},\omega_{t})\in S_{\HH^n}$, $dS_{\HH^n}(\omega)$ is the surface measure on $S_{\HH^n}$, and the function
$$\psi(\sigma)=\int_{\{\omega\in S_{\HH^n}:-\omega_{x_1}>\sigma d(\omega)\}}dS_{\HH^n}(\omega)>0\ \ \forall\,\sigma\in(0,1).$$
Inequality \eqref{eq 3.2} follows from \eqref{eq 3.3} and \eqref{eq 3.4} with $C=\frac{1}{p}\underset{0<\sigma<1}{\mathrm{sup}}\left[\sigma^p \psi(\sigma)\right]$.
\end{proof}
The following one-dimensional result is used to prove a useful result (Lemma \ref{lemma 3.5}) for the distance function $\delta_{\HH^n_+}$. The proof can be found in \cite{IaSu}, (see Lemma 3.1).
\begin{lemma}\label{lemma 3.4}
Let $1<p<\infty$ and $0<s<1$. Then the function
\begin{equation}\label{eq 3.5}
g_\epsilon^{(1)}(x):=\int_{B_\epsilon(x)^C}\frac{J_p(x^s-y_+^s)}{|x-y|^{1+sp}}dy\rightarrow0
\end{equation}
uniformly on any compact subset $K\subset\RR_+$ as $\epsilon\rightarrow0$.
\end{lemma}
\begin{lemma}\label{lemma 3.5}
Let $1<p<\infty$ and $0<s<1$. For any fixed $\xi\in\HH^n_+$ define the integral
$$g_\epsilon(\xi):=\int_{B(\xi,\epsilon)^C}\frac{J_p(x_1^s-{x_1'}_+^s)}{d({\xi}^{-1}\circ \xi')^{Q+sp}}d\xi'.$$
Then $\underset{\epsilon\rightarrow0}{\lim}\,g_\epsilon(\xi)=0$ strongly in $L^1_{\mathrm{loc}}(\HH^n_+)$.
\end{lemma}
\begin{proof}
The change of variable $\tilde{\xi}={\xi}^{-1}\circ \xi'$ gives
$$g_\epsilon(\xi)=\int_{B(0,\epsilon)^C}\frac{J_p(x_1^s-({x_1+\Tilde{x}_1})_+^s)}{d(\Tilde{\xi})^{Q+sp}}d\Tilde{\xi}=\int_{S_{\HH^n}}\int_\epsilon^\infty\frac{J_p(x_1^s-(x_1+r\omega_{x_1})_+^s)}{r^{Q+sp}}r^{Q-1}drd\omega.$$
Here $\omega=(\omega_{x_1},\ldots,\omega_{x_n},\omega_{y_1},\ldots,\omega_{y_n},\omega_{t})\in S_{\HH^n}$. The last integral can be written as
$$g_\epsilon(\xi)=\int_{S_{\HH^n}\cap\,\HH^n_+}\omega_{x_1}^{1+sp}\int_{(-\epsilon,\epsilon)^C}\frac{J_p(x_1^s-(x_1+r\omega_{x_1})_+^s)}{(r\omega_{x_1})^{1+sp}}drd\omega=\int_{S_{\HH^n}\cap \,\HH^n_+}\omega_{x_1}^{sp}g_{\epsilon\omega_{x_1}}^{(1)}(x_1)d\omega,$$
where the function $g_{\epsilon\omega_{x_1}}^{(1)}$ is defined in \eqref{eq 3.5}. The conclusion follows from Lemma \ref{lemma 3.4} and the fact $$\int_{S_{\HH^n}\cap \,\HH^n_+}\omega_{x_1}^{sp}d\omega<\infty.$$
\end{proof}\smallskip

The next result shows that $(\delta_{\HH^n_+})^s$, the $s$-th power of the distance function, is weakly harmonic.

\begin{proposition}\label{prop 3.1}
Let $1<p<\infty$, and $0<s<1$. We have $(\delta_{\HH^n_+})^s\in W_{\mathrm{loc}}^{s,p}(\HH^n_+)\cap L_{sp}^{p-1}(\HH^n)$ and
\begin{equation}
\int_{\HH^n}\int_{\HH^n}\frac{J_p(\delta_{\HH^n_+}(\xi)^s-\delta_{\HH^n_+}(\xi')^s)}{d({\xi}^{-1}\circ \xi')^{Q+sp}}(\phi(\xi)-\phi(\xi'))d\xi'd\xi=0,
\end{equation}
for every non-negative $\phi\in W^{s,p}(\HH^n_+)$ with compact support in $\HH^n_+$.
\end{proposition}
\begin{proof}
We have
$$\int_{\HH^n}\frac{|\delta_{\HH^n_+}(\xi)^s|^{p-1}}{(1+d(\xi))^{Q+sp}}d\xi=\int_{\HH^n}\frac{{x_1}_+^{s(p-1)}}{(1+d(\xi))^{Q+sp}}d\xi\leq \int_{\HH^n}\frac{d(\xi)^{s(p-1)}}{(1+d(\xi))^{Q+sp}}d\xi$$
Using the polar coordinate for $\HH^n$, we get
$$\int_{\HH^n}\frac{|\delta_{\HH^n_+}(\xi)^s|^{p-1}}{(1+d(\xi))^{Q+sp}}d\xi\leq|S_{\HH^n}|\int_0^\infty\frac{r^{Q-1+s(p-1)}}{(1+r)^{Q+sp}}dr\leq |S_{\HH^n}|\int_0^\infty\frac{dr}{(1+r)^{1+s}}<\infty.$$
To show $(\delta_{\HH^n_+})^s\in W_{\mathrm{loc}}^{s,p}(\HH^n_+)$, let us consider any compact subset $K$ of $\HH^n_+$. We have
\begin{equation}\label{eq 3.7}
[(\delta_{\HH^n_+})^s]^p_{s,p,K}=\int_K\int_K \frac{|x_1^s-{x_1'}^s|^p}{d({\xi}^{-1}\circ \xi')^{Q+sp}}d\xi'd\xi.
\end{equation}
By Mean value theorem in $\RR$, $\exists\,\tilde{x}_1\in(x_1,x_1')$ when $x_1<x_1'$, or $\Tilde{x}_1\in(x_1',x_1)$ when $x_1'<x_1$ such that
\begin{equation}\label{eq 3.8}
|x_1^s-{x_1'}^s|=s\tilde{x}_1^{s-1}|x_1-x_1'|\leq sr^{s-1}d({\xi}^{-1}\circ \xi'),
\end{equation}
where $r=\inf\{x_1:\xi=(x_1,\cdots,x_n,y,t)\in K\}>0$. Using \eqref{eq 3.8} in \eqref{eq 3.7} we obtain
$$[(\delta_{\HH^n_+})^s]^p_{s,p,K}\leq s^pr^{p(s-1)}\int_K\int_K \frac{d\xi'}{d({\xi}^{-1}\circ \xi')^{Q-p+sp}}d\xi.$$
Let $R>0$ be such that $K\subset B(\xi,R)$ $\forall\, \xi\in K$. A change of variable gives
$$[(\delta_{\HH^n_+})^s]^p_{s,p,K}\leq s^pr^{p(s-1)}|K|\int_{B(0,R)} \frac{d\xi'}{d(\xi')^{Q-p+sp}}=\frac{s^pr^{p(s-1)}|K||S_{\HH^n}|}{p-sp}R^{p-sp}<\infty.$$\smallskip

The function
$$(\xi,\xi')\mapsto\int_{\HH^n}\int_{\HH^n}\frac{J_p(\delta_{\HH^n_+}(\xi)^s-\delta_{\HH^n_+}(\xi')^s)}{d({\xi}^{-1}\circ \xi')^{Q+sp}}(\phi(\xi)-\phi(\xi'))$$
is integrable by Lemma \ref{Lemma 2.1}. Applying dominated convergence theorem we can write
\begin{multline*}
J=\int_{\HH^n}\int_{\HH^n}\frac{J_p(\delta_{\HH^n_+}(\xi)^s-\delta_{\HH^n_+}(\xi')^s)}{d({\xi}^{-1}\circ \xi')^{Q+sp}}(\phi(\xi)-\phi(\xi'))d\xi'd\xi\\
=\lim_{\epsilon\rightarrow0}\iint_{\mathcal{T}_\epsilon}\frac{J_p(\delta_{\HH^n_+}(\xi)^s-\delta_{\HH^n_+}(\xi')^s)}{d({\xi}^{-1}\circ \xi')^{Q+sp}}(\phi(\xi)-\phi(\xi'))d\xi'd\xi\qquad\qquad\qquad\qquad\qquad\qquad\quad\\
=\lim_{\epsilon\rightarrow0}\iint_{\mathcal{T}_\epsilon}\frac{J_p(\delta_{\HH^n_+}(\xi)^s-\delta_{\HH^n_+}(\xi')^s)}{d({\xi}^{-1}\circ \xi')^{Q+sp}}\phi(\xi)d\xi'd\xi-\lim_{\epsilon\rightarrow0}\iint_{\mathcal{T}_\epsilon}\frac{J_p(\delta_{\HH^n_+}(\xi)^s-\delta_{\HH^n_+}(\xi')^s)}{d({\xi}^{-1}\circ \xi')^{Q+sp}}\phi(\xi')d\xi'd\xi,
\end{multline*}
where the set $\mathcal{T}_\epsilon$ is defined in Lemma \ref{lemma 3.1}. The last integral on the right-hand side is integrable, thanks to Lemma \ref{lemma 3.1}. Let $\mathcal{O}$ be the support of $\phi$. Applying Fubini's theorem in the last integral and then changing the role of $\xi$ and $\xi'$ we get
\begin{multline}\label{eq 3.9}
J=2\lim_{\epsilon\rightarrow0}\iint_{\mathcal{T}_\epsilon}\frac{J_p(\delta_{\HH^n_+}(\xi)^s-\delta_{\HH^n_+}(\xi')^s)}{d({\xi}^{-1}\circ \xi')^{Q+sp}}\phi(\xi)d\xi'd\xi\\
=2\lim_{\epsilon\rightarrow0}\int_{\mathcal{O}}\left(\int_{B(\xi,\epsilon)^C}\frac{J_p(x_1^s-{x_1'}_+^s)}{d({\xi}^{-1}\circ \xi')^{Q+sp}}d\xi'\right)\phi(\xi)d\xi.
\end{multline}
Applying Lemma \ref{lemma 3.5} in \eqref{eq 3.9}, we obtain
$$J=2\int_{\mathcal{O}}\left(\lim_{\epsilon\rightarrow0}\int_{B(\xi,\epsilon)^C}\frac{J_p(x_1^s-{x_1'}_+^s)}{d({\xi}^{-1}\circ \xi')^{Q+sp}}d\xi'\right)\phi(\xi)d\xi=0.$$
This completes the proof.
\end{proof}

\section{Proof of the fractional Hardy's inequality}\label{sec 4}

Before proving the main results, let us mention some fundamental inequalities which we are going to use. The proofs are given in \cite{BraCin}.
\begin{lemma}\label{lemma 4.1}
Let $1<p<\infty$ and $a,b,c,d\in\RR$ with $a,b>0$ and $c,d\geq0$. Then there exist two positive constants $C_1$ and $C_2$ depends only on $p$, such that
\begin{equation}\label{eq 4.1}
J_p(a-b)\left(\frac{c^p}{a^{p-1}}-\frac{d^p}{b^{p-1}}\right)+C_1\left|\frac{a-b}{a+b}\right|^p(c^p+d^p)\leq C_2|c-d|^p.
\end{equation}
\end{lemma}
\begin{lemma}\label{lemma 4.2}
Let $0<s<1$. Then for every $a,b>0$ we have
\begin{equation}\label{eq 4.2}
\frac{|a^s-b^s|}{a^s+b^s}\geq\frac{s}{2}\frac{a-b}{\mathrm{max}\{a,b\}}.
\end{equation}
\end{lemma}

\begin{proof}[\textbf{Proof of Theorem \ref{thm 1.1}}]
As like in the proof of the Euclidean result, using the results from the previous section, we first establish the inequality \eqref{eq main} with the constant $\frac{s^pC_1(n,p)}{(1-s)}$. Since this constant does not behave well when $s$ is very small, in the second part we establish the same inequality with the improved constant $\frac{C_2(n,p)}{s}$ for small values of $s$. Together, these two results imply that the inequality \eqref{eq main} holds with the constant $\frac{s^{p+1}C_1(n,p)+(1-s)C_2(n,p)}{2s(1-s)}$. We conclude the proof by noting that the quantity $s^{p+1}C_1(n,p)+(1-s)C_2(n,p)$ is bounded from below by a constant $C_(n,p)>0$ (see \cite{BraCin}, Remark 4.1), which is independent of $s$.\smallskip

\noindent \textbf{Proof of the first part:} By Proposition \ref{prop 3.1} we have
\begin{equation}\label{eq 4.3}
\int_{\HH^n}\int_{\HH^n}\frac{J_p(\delta_{\HH^n_+}(\xi)^s-\delta_{\HH^n_+}(\xi')^s)}{d({\xi}^{-1}\circ \xi')^{Q+sp}}(\phi(\xi)-\phi(\xi'))d\xi'd\xi=0,
\end{equation}
for every non-negative $\phi\in W^{s,p}(\HH^n_+)$ with compact support in $\HH^n_+$. Let $f\in C_c^\infty(\HH^n_+)$ and $\mathcal{O}$ be the support of $f$. For $\epsilon>0$ let us define the function
\begin{equation}\label{eq 4.4}
\phi=\frac{|f|^p}{((\delta_{\HH^n_+})^s+\epsilon)^{p-1}}.
\end{equation}
Clearly, $((\delta_{\HH^n_+})^s+\epsilon)^{1-p}\in L^\infty(\HH^n_+)$. Observe that the function $g(t)=(t+\epsilon)^{1-p}$ for $t>0$ is Lipschitz. Now for any compact subset $K$ of $\HH^n_+$,
\begin{multline*}
[((\delta_{\HH^n_+})^s+\epsilon)^{1-p}]^p_{s,p,K}\\
=\int_K\int_K\frac{|g(\delta_{\HH^n_+}(\xi)^s)-g(\delta_{\HH^n_+}(\xi)^s)|^p}{d({\xi}^{-1}\circ \xi')^{Q+sp}}d\xi'd\xi\leq C\int_K\int_K\frac{|\delta_{\HH^n_+}(\xi)^s-\delta_{\HH^n_+}(\xi')^s|^p}{d({\xi}^{-1}\circ \xi')^{Q+sp}}d\xi'd\xi.
\end{multline*}
The last integral is finite since $(\delta_{\HH^n_+})^s\in W^{s,p}_{\mathrm{loc}}(\HH^n_+)$, and this implies $(\delta_{\HH^n_+}^s+\epsilon)^{1-p}\in W^{s,p}_{\mathrm{loc}}(\HH^n_+)$. Hence the function $\phi$ defined in \eqref{eq 4.4} is admissible in \eqref{eq 4.3}, thanks to Lemma \ref{lemma 3.2} and Remark \ref{remark 2}.\smallskip

\noindent Decomposing the integral of \eqref{eq 4.3} we get
\begin{multline}\label{eq 4.5}
0=\int_{\HH^n_+}\int_{\HH^n_+}\frac{J_p(\delta_{\HH^n_+}(\xi)^s-\delta_{\HH^n_+}(\xi')^s)}{d({\xi}^{-1}\circ \xi')^{Q+sp}}\left(\frac{|f(\xi)|^p}{(\delta_{\HH^n_+}(\xi)^s+\epsilon)^{p-1}}-\frac{|f(\xi')|^p}{(\delta_{\HH^n_+}(\xi')^s+\epsilon)^{p-1}}\right)d\xi'd\xi\\
+2\int_{\mathcal{O}}\int_{{\HH^n_+}^C}\frac{J_p(\delta_{\HH^n_+}(\xi)^s-\delta_{\HH^n_+}(\xi')^s)}{d({\xi}^{-1}\circ \xi')^{Q+sp}}\frac{|f(\xi)|^p}{(\delta_{\HH^n_+}(\xi)^s+\epsilon)^{p-1}}d\xi'd\xi.
\end{multline}
One can observe that
$$\frac{J_p(\delta_{\HH^n_+}(\xi)^s-\delta_{\HH^n_+}(\xi')^s)}{(\delta_{\HH^n_+}(\xi)^s+\epsilon)^{p-1}}\leq1,$$
which gives
\begin{equation}\label{eq 4.6}
\int_{\mathcal{O}}\int_{{\HH^n_+}^C}\frac{J_p(\delta_{\HH^n_+}(\xi)^s-\delta_{\HH^n_+}(\xi')^s)}{d({\xi}^{-1}\circ \xi')^{Q+sp}}\frac{|f(\xi)|^p}{(\delta_{\HH^n_+}(\xi)^s+\epsilon)^{p-1}}d\xi'd\xi\leq \int_{\mathcal{O}}\int_{{\HH^n_+}^C}\frac{|f(\xi)|^p}{d({\xi}^{-1}\circ \xi')^{Q+sp}}d\xi'd\xi.
\end{equation}
Now, to estimate the first integral of \eqref{eq 4.5}, let us consider
$$a=\delta_{\HH^n_+}(\xi)^s+\epsilon,\ \ b=\delta_{\HH^n_+}(\xi')^s+\epsilon,\ \ c=|f(\xi)|,\ \ \mathrm{and}\ \ d=|f(\xi')|$$
in Lemma \ref{lemma 4.1}. Dividing both sides of \eqref{eq 4.1} by $d({\xi}^{-1}\circ \xi')^{Q+sp}$ and then integrating twice with respect to $\xi'$ and $\xi$ respectively on $\HH^n_+$, we get
\begin{multline}\label{eq 4.7}
I=\int_{\HH^n_+}\int_{\HH^n_+}\frac{J_p(\delta_{\HH^n_+}(\xi)^s-\delta_{\HH^n_+}(\xi')^s)}{d({\xi}^{-1}\circ \xi')^{Q+sp}}\left(\frac{|f(\xi)|^p}{(\delta_{\HH^n_+}(\xi)^s+\epsilon)^{p-1}}-\frac{|f(\xi')|^p}{(\delta_{\HH^n_+}(\xi')^s+\epsilon)^{p-1}}\right)d\xi'd\xi\\
\leq-C_1\int_{\HH^n_+}\int_{\HH^n_+}\left|\frac{\delta_{\HH^n_+}(\xi)^s-\delta_{\HH^n_+}(\xi')^s}{\delta_{\HH^n_+}(\xi)^s+\delta_{\HH^n_+}(\xi')^s+2\epsilon}\right|^p\left(|f(\xi)|^p+|f(\xi')|^p\right)\frac{d\xi'd\xi}{d({\xi}^{-1}\circ \xi')^{Q+sp}}\\
+C_2\int_{\HH^n_+}\int_{\HH^n_+}\frac{||f(\xi)|-|f(\xi)||^p}{d({\xi}^{-1}\circ \xi')^{Q+sp}}d\xi'd\xi.
\end{multline}
Applying \eqref{eq 4.6} and \eqref{eq 4.7} in \eqref{eq 4.5}, we obtain
\begin{equation}\label{eq 4.8}
C_1\int_{\HH^n_+}\int_{\HH^n_+}\left|\frac{x_1^s-{x_1'}^s}{x_1^s+{x_1'}^s}\right|^p\left(|f(\xi)|^p+|f(\xi')|^p\right)\frac{d\xi'd\xi}{d({\xi}^{-1}\circ \xi')^{Q+sp}}\leq C_2[f]^p_{s,p,\HH^n}.
\end{equation}
Here, we pass the limit $\epsilon\rightarrow0$ and apply Fatou's lemma. We have by symmetry,
\begin{multline*}
J=\int_{\HH^n_+}\int_{\HH^n_+}\left|\frac{x_1^s-{x_1'}^s}{x_1^s+{x_1'}^s}\right|^p\left(|f(\xi)|^p+|f(\xi')|^p\right)\frac{d\xi'd\xi}{d({\xi}^{-1}\circ \xi')^{Q+sp}}\\
=2\int_{\HH^n_+}\int_{\HH^n_+}\left|\frac{x_1^s-{x_1'}^s}{x_1^s+{x_1'}^s}\right|^p|f(\xi)|^p\frac{d\xi'd\xi}{d({\xi}^{-1}\circ \xi')^{Q+sp}}\\
\geq2\int_{\HH^n_+}\left(\int_{\{\xi'\in\HH^n_+:x_1'\leq x_1\}}\left|\frac{x_1^s-{x_1'}^s}{x_1^s+{x_1'}^s}\right|^p\frac{d\xi'}{d({\xi}^{-1}\circ \xi')^{Q+sp}}\right)|f(\xi)|^pd\xi.
\end{multline*}
Using Lemma \ref{lemma 4.2} with $a=x_1$ and $b=x_1'$ and then by Lemma \ref{lemma 3.3} we obtain
\begin{equation}\label{eq 4.9}
J\geq \frac{s^p}{2^{p-1}}\int_{\HH^n_+}\left(\int_{\{\xi'\in\HH^n_+:x_1'\leq x_1\}}\frac{|x_1-x_1'|^p}{d({\xi}^{-1}\circ \xi')^{Q+sp}}d\xi'\right)\frac{|f(\xi)|^p}{x_1^p}d\xi\geq \frac{s^p C}{2^{p-1}(1-s)}\int_{\HH^n_+}\frac{|f(\xi)|^p}{x_1^{sp}}d\xi.
\end{equation}
Inequality \eqref{eq main} with the constant $\frac{s^pC_1(n,p)}{1-s}$ follows from \eqref{eq 4.8} and \eqref{eq 4.9}. This completes the proof of the first part.\smallskip

\noindent \textbf{Proof of the second part:} For every $\xi\in \HH^n_+$, let us define the set
$$\Omega(\xi,r)=\{\xi'\in B(\xi,2r)^C:d({\xi}^{-1}\circ \xi')<2(x_1-x_1')\}.$$
If $\xi'\in\Omega(\xi,x_1)$, then $2x_1\leq d({\xi}^{-1}\circ \xi')<2(x_1-x_1')$, which implies $x_1'<0$. Therefore, $\Omega(\xi,x_1)\subset{\HH^n_+}^C$ and we have for $f\in C_c^\infty(\HH^n_+)$,
\begin{multline*}
\int_{\HH^n}\int_{\HH^n}\frac{|f(\xi)-f(\xi')|^p}{d({\xi}^{-1}\circ \xi')^{Q+sp}}d\xi' d\xi
\geq \int_{\HH^n_+}|f(\xi)|^p\left(\int_{{\HH^n_+}^C}\frac{d\xi'}{d({\xi}^{-1}\circ \xi')^{Q+sp}}\right)d\xi\\
\geq \int_{\HH^n_+}|f(\xi)|^p\left(\int_{\Omega(\xi,x_1)}\frac{d\xi'}{d({\xi}^{-1}\circ \xi')^{Q+sp}}\right)d\xi=\int_{\HH^n_+}|f(\xi)|^p\left(\int_{\Omega(0,x_1)}\frac{d\xi'}{d(\xi')^{Q+sp}}\right)d\xi.
\end{multline*}
Using the polar coordinates for $\HH^n$ with $\omega=(\omega_{x_1},\ldots,\omega_{x_n},\omega_{y_1},\ldots,\omega_{y_n},\omega_{t})\in S_{\HH^n}$, we obtain
\begin{multline*}
=\int_{\HH^n_+}|f(\xi)|^p\left(\int_{\{\omega\in S_{\HH^n}:d(\omega)<-2\omega_{x_1}\}}dS_{\HH^n}(\omega)\right)\left(\int_{2x_1}^\infty r^{-sp-1}dr\right)d\xi\\
=C\frac{1}{sp}\frac{1}{2^{sp}}\int_{\HH^n_+}\frac{|f(\xi)|^p}{x_1^{sp}}d\xi\geq\frac{C}{p2^p}\frac{1}{s}\int_{\HH^n_+}\frac{|f(\xi)|^p}{x_1^{sp}}d\xi.
\end{multline*}
This completes the proof.
\end{proof}

Next, we prove Theorem \ref{thm alpha>0}. The proof follows from Theorem \ref{thm 1.1}, i.e., from the case $\alpha=0$.

\begin{proof}[\textbf{Proof of Theorem \ref{thm alpha>0}.}]
Let $g\in C_c^\infty(\HH^n_+)$. Clearly, for $\alpha\geq0$,
$$\int_{\HH^n}\int_{\HH^n}\frac{|g(\xi)-g(\xi')|^p}{d({\xi}^{-1}\circ \xi')^{Q+sp+\alpha}}d\xi'd\xi\leq \int_{\HH^n}\int_{\HH^n}\frac{|g(\xi)-g(\xi')|^p}{d({\xi}^{-1}\circ \xi')^{Q+sp}|z'-z|^\alpha}d\xi'd\xi.$$
Let us consider $s'=s$ and $p'=p+\frac{\alpha}{s}$, when $s\geq\frac{1}{p}$, and $s'=\frac{1}{1+\alpha}\left(s+\frac{\alpha}{p}\right)$ and $p'=p(1+\alpha)$, when $s<\frac{1}{p}$. Note that $p\leq p'\leq p(1+\alpha)$, $0<s\leq s'<1$, and $s'p'=sp+\alpha$. To prove \eqref{eq alpha>0}, it is enough to prove the inequality
\begin{equation}\label{eq 4.10}
\frac{C(n,p,\alpha)}{(s+\frac{\alpha}{p})(1-s)}\int_{\HH^n_+}\frac{|g(\xi)|^p}{x_1^{s'p'}}d\xi\leq\int_{\HH^n}\int_{\HH^n}\frac{|g(\xi)-g(\xi')|^p}{d({\xi}^{-1}\circ \xi')^{Q+s'p'}}d\xi'd\xi.
\end{equation}
Let $f(\xi)=|g(\xi)|^{\frac{p}{p'}}$. The elementary inequality $(a+b)^t\leq a^t+b^t$ for $a,b\geq0$ and $0<t\leq1$ gives
\begin{equation}\label{eq fg}
|f(\xi)-f(\xi')|^{p'}\leq |f(\xi)-f(\xi')|^p.
\end{equation}
Here, we consider $f(\xi)\geq f(\xi')$ (by symmetry), $t=p/p'$, $a=f(\xi)^{(1/t)}-f(\xi')^{(1/t)}$, and $b=f(\xi')^{(1/t)}$. We also have,
\begin{equation}\label{eq constant}
 \frac{C(n,p')}{(s')(1-s')}\geq\frac{C(n,p,\alpha)}{(s+\frac{\alpha}{p})(1-s)}.
\end{equation}
By a density argument, \eqref{eq main} is also true for the functions in $C_c(\HH^n_+)$. Since $f\in C_c(\HH^n_+)$, we obtain from inequality \eqref{eq main} for $s'$ and $p'$ defined above,
$$\frac{C(n,p')}{s'(1-s')}\int_{\HH^n_+}\frac{|f(\xi)|^{p'}}{x_1^{s'p'}}d\xi\leq \int_{\HH^n}\int_{\HH^n}\frac{|f(\xi)-f(\xi')|^{p'}}{d({\xi}^{-1}\circ \xi')^{Q+s'p'}}d\xi' d\xi$$
Therefore, inequality \eqref{eq 4.10} follows, if we put $f=|g|^{p/p'}$ above, and use \eqref{eq fg} and \eqref{eq constant}.
\end{proof}

\section{Sharp constant for the fractional Hardy's inequality}\label{sec 5}
Although the general inequality \eqref{eq alpha>0} follows from \eqref{eq main}, that is, from case $\alpha=0$, the computation of the sharp constant for \eqref{eq alpha>0} may not necessarily follow from the sharp constant results for \eqref{eq main}. Before presenting the results for the Heisenberg group, let us first discuss some one-dimensional results relevant to our proofs. For any $\beta\in\RR$, let us define the function
$$u_\beta(\tau)=\tau^\beta,\ \ \tau\in\RR_+.$$
For $0<\epsilon\ll1$, let us define the following set in $\RR_+$:
\begin{equation}\label{set I_ep}
I_\epsilon(t)=\left(\frac{t}{1+\epsilon},(1+\epsilon)t\right).
\end{equation}
\begin{proposition}\label{prop 5.1}
\emph{(i)} Let $1<p<\infty$, $0<s<1$, and $-\frac{1}{p-1}<\beta<\frac{sp}{p-1}$, and we define the quantities
\begin{equation}
\lambda(\beta,s,p)=2\int_0^1\frac{J_p(1-\tau^\beta)}{(1-\tau)^{1+sp}}(1-\tau^{sp-1-\beta(p-1)})d\tau+\frac{2}{sp},
\end{equation}
and
\begin{equation}
\lambda_\epsilon(\beta,s,p)=2\int_0^{\frac{1}{1+\epsilon}}\frac{J_p(1-\tau^\beta)}{(1-\tau)^{1+sp}}d\tau+2\int_{1+\epsilon}^\infty\frac{J_p(1-\tau^\beta)}{(1-\tau)^{1+sp}}d\tau+\frac{2}{sp}.
\end{equation}
We have
\begin{equation}\label{eq beta}
\lim_{\epsilon\rightarrow0}\lambda_\epsilon(\beta,s,p)=\lambda(\beta,s,p),
\end{equation}
and the function $u_\beta(\tau)$ is a weak solution of the equation
\begin{equation}\label{eq p-lap}
(-\Delta_p)^su=\lambda\frac{J_p(u)}{\delta_{\RR_+}^{sp}} \ \ \mathrm{in}\ \RR_+,
\end{equation}
with $\lambda=\lambda(\beta,s,p)$. Moreover, we have
\begin{equation}\label{eq F-ep}
F_\epsilon(\tau):=2\int_{\RR\setminus I_\epsilon(\tau)}\frac{J_p(u_\beta(\tau)-u_\beta(\tau'))}{|\tau-\tau'|^{1+sp}}d\tau'=\lambda_\epsilon(\beta,s,p)\frac{u_\beta(\tau)^{p-1}}{\tau^{sp}},\ \ \forall\,\tau\in\RR_+,\ 0<\epsilon<1,
\end{equation}
and this family of functions converges to
$$F_0(\tau)=\lambda(\beta,s,p)\frac{u_\beta(\tau)^{p-1}}{\tau^{sp}},$$
uniformly on compact subsets of $\RR_+$, as $\epsilon$ goes to $0$.\smallskip

\noindent \emph{(ii)} The function $\beta\mapsto\lambda(\beta,s,p)$ is maximal at $\frac{sp-1}{p}$, i.e.,
\begin{equation}
\lambda(\beta,s,p)\leq \lambda\left(\frac{sp-1}{p},s,p\right)=2\int_0^1\frac{\left|1-\tau^{\frac{sp-1}{p}}\right|^p}{(1-\tau)^{1+sp}}d\tau+\frac{2}{sp}.
\end{equation}
\end{proposition}
The fractional Laplacian operator $(-\Delta_p)^s$, the notion of a weak solution of \eqref{eq p-lap}, and the proof of the above proposition are given in \cite{Brasco2}. The following result is useful for converting certain integrals over the Heisenberg group into one-dimensional integrals.
\begin{lemma}\label{lemma 5.1}
Let $n\geq1$, $\theta>0$, and $\alpha\geq0$. We have for every $m>0$
\begin{equation}\label{eq 5.7}
\int_{\RR^{2n-1}}\int_{-\infty}^\infty\frac{d\tilde{x}dt}{((m^2+|\tilde{x}|^2)^2+t^2)^{\frac{Q+\theta}{4}}(m^2+|\tilde{x}|^2)^{\frac{\alpha}{2}}}=\frac{2|S_{2n-2}|}{m^{1+\theta+\alpha}}\mathcal{I}(Q-2,\theta,\alpha),
\end{equation}
where $\tilde{x}\in\RR^{2n-1}$, $S_{2n-2}$ is the unit sphere in $\RR^{2n-1}$, and for $k\in\NN$ with $k\geq2$, the quantity $\mathcal{I}(k,\theta,\alpha)$ is defined in \eqref{eq I} in Introduction.
\end{lemma}
\begin{proof}
We have
\begin{multline*}
I=\int_{\RR^{2n-1}}\int_{-\infty}^\infty\frac{d\tilde{x}dt}{((m^2+|\tilde{x}|^2)^2+t^2)^{\frac{Q+\theta}{4}}(m^2+|\tilde{x}|^2)^{\frac{\alpha}{2}}}\\
=\frac{1}{m^{Q+\theta+\alpha}}\int_{\RR^{2n-1}}\int_{-\infty}^\infty\frac{d\tilde{x}dt}{((1+|\frac{\tilde{x}}{m}|^2)^2+\frac{t^2}{m^4})^{\frac{Q+\theta}{4}}(1+|\frac{\tilde{x}}{m}|^2)^{\frac{\alpha}{2}}}.
\end{multline*}
Here $\frac{\tilde{x}}{m}=(\frac{\tilde{x}_1}{m},\frac{\tilde{x}_2}{m},\ldots,\frac{\tilde{x}_{2n-1}}{m})$. Using the changes of variable $\frac{\tilde{x}}{m}=\tilde{x}'$ and $\frac{t}{m^2}=t'$, we get
\begin{multline*}
I=\frac{1}{m^{1+\theta+\alpha}}\int_{\RR^{2n-1}}\int_{-\infty}^\infty\frac{d\tilde{x}'dt'}{((1+|\tilde{x}'|^2)^2+t'^2)^{\frac{Q+\theta}{4}}(1+|\tilde{x}'|^2)^{\frac{\alpha}{2}}}\\
=\frac{1}{m^{1+\theta+\alpha}}\int_{\RR^{2n-1}}\left(\int_{-\infty}^\infty\frac{dt'}{\left(1+\frac{t'^2}{(1+|\tilde{x}'|^2)^2}\right)^{\frac{Q+\theta}{4}}}\right)\frac{d\tilde{x}'}{(1+|\tilde{x}'|^2)^{\frac{Q+\theta+\alpha}{2}}}.
\end{multline*}
Again, using the change of variable $t=t'/(1+|\tilde{x}'|^2)$, we obtain
$$I=\frac{1}{m^{1+\theta+\alpha}}\int_{-\infty}^\infty\frac{dt}{(1+t^2)^{\frac{Q+\theta}{4}}}\int_{\RR^{2n-1}}\frac{d\tilde{x}'}{(1+|\tilde{x}'|^2)^{\frac{Q-2+\theta+\alpha}{2}}}.$$
The final conclusion follows by converting the last integral to polar coordinates in $\RR^{2n-1}$.
\end{proof}

For any $\beta\in\RR$, let us define the following function in $\HH^n$:
$$u_{\beta}(\xi):=\delta_{\HH^n_+}(\xi)^\beta={x_1}_+^\beta.$$
Since the function $u_\beta$ does not belong to the space $W^{s,p,\alpha}_{\mathrm{loc}(x_1)}(\HH^n_+)$, we work with the cut-off function defined below. For any $R_0>0$, we define the following sets:
$$D_{R_0}:=\{\xi=(x,y,t)\in\HH^n:|x_i|,|y_i|\leq R_0,|t|\leq R_0^2\},$$
and
$$D_{R_0,t}:=\{\xi=(x,y,t)\in\HH^n:|x_i|,|y_i|\leq R_0,|t|\leq R_0^2+R_0\},$$
and define the function
$$u_{\beta,R_0}=u_\beta\psi_{R_0},$$
where $\psi_{R_0}\in C_c^\infty(D_{R_0+1,t})$ be a function satisfies $0\leq\psi_{R_0}\leq1$, $\psi\equiv1$ in $D_{R_0}$, $|\frac{\partial\psi_{R_0}}{\partial x_i}|,|\frac{\partial\psi_{R_0}}{\partial y_i}|\leq C_0$, and $|\frac{\partial\psi_{R_0}}{\partial t}|\leq\frac{C_0}{R_0+1}$ for some constant $C_0>0$, independent of $R_0$. The following lemma is useful to show that the function $u_{\beta,R_0}\in W^{s,p,\alpha}_{\mathrm{loc}(x_1)}(\HH^n_+)$.
\begin{lemma}\label{lemma 5.2}
Let $R_0>0$ and $0<r<R$. Then there exists a constant $C(r,R)>0$, independent of $R_0$, such that
$$|\nabla_{\HH^n}\,u_{\beta,R_0}(\xi)|\leq C(r,R),$$
for all $\xi\in\mathscr{D}_{r,R}\cap D_{R_0+1,t}$. Moreover, $|\nabla_{\HH^n}\,u_{\beta,R_0}(\xi)|$ is zero outside $D_{R_0+1,t}$.
\end{lemma}
\begin{proof}
Since the support of $u_{\beta,R_0}$ is $D_{R_0+1,t}$, $|\nabla_{\HH^n}\,u_{\beta,R_0}(\xi)|=0$ outside $D_{R_0+1,t}$. Let $\xi\in\mathscr{D}_{r,R}$, we have
\begin{multline*}
|X_1u_{\beta,R_0}(\xi)|=\left|\left(\frac{\partial}{\partial x_1}+2y_1\frac{\partial}{\partial t}\right)(u_\beta\psi_{R_0})\right|=\left|\beta x_1^{\beta-1}\psi_{R_0}+u_\beta\frac{\partial\psi_{R_0}}{\partial x_1}+2y_1u_\beta\frac{\partial\psi_{R_0}}{\partial t}\right|\\
\leq|\beta|\max\{r^{\beta-1},R^{\beta-1}\}+\max\{r^{\beta},R^{\beta}\}\left(C_0+2(R_0+1)\frac{C_0}{R_0+1}\right)\\
=|\beta|\max\{r^{\beta-1},R^{\beta-1}\}+3C_0\max\{r^{\beta},R^{\beta}\},
\end{multline*}
and for $2\leq i\leq n$, and $1\leq j\leq n$,
$$|X_iu_{\beta,R_0}(\xi)|,|Y_ju_{\beta,R_0}(\xi)|\leq3C_0\max\{r^{\beta},R^{\beta}\}.$$
These conclude the proof.
\end{proof}
Before proving $u_{\beta,R_0}\in W^{s,p,\alpha}_{\mathrm{loc}(x_1)}(\HH^n_+)$, let us recall the following useful result about the structure of the Heisenberg group. The proof is given in \cite[pp. 727]{Bonfiglioli} for any general Homogeneous Carnot Group.
\begin{theorem}\label{thm 5.2}
Let $g$ denote the Lie algebra of vector fields of $\HH^n$, and let $\mathrm{Exp}:g\rightarrow\HH^n$ be the usual exponential map. Then $\exists$ $M\in\NN$, depends only on $\HH^n$, such that for any $h\in\HH^n$,
\begin{equation}\label{eq h}
h=h_1\circ\ldots\circ h_M,
\end{equation}
\begin{equation}
d(h_j)\leq c_0d(h),\ \ \mathrm{for}\ j=1,\ldots,M,
\end{equation}
where $c_0>0$ is a constant independent of $h$ and $h_j$, $h_j=\mathrm{Exp}(t_jX_{i_j})=t_je_{i_j}$, for some $t_j\in\RR$, $i_j\in(1,\ldots,2n)$, $\{e_i\}_{i=1}^{2n+1}$ is the standard basis on $\RR^{2n+1}$, and $d(h_j)=|t_j|$. Here $X_i=Y_{i-n}$ for $i=n+1,\ldots,2n$.
\end{theorem}
\begin{remark}\label{remark 3}
\emph{Let $h=(h_{x_1},\ldots,h_{y_n},h_t)\in\{\xi\in\HH^n:x_1>0\}$ and $h_j=(h_{j,x_1},\ldots,h_{j,y_n},h_{j,t})$. If $h_j=t_je_{i_j}$, with $e_{i_j}\neq e_1$, we must have $h_{j,x_1}=0$. Now, let $h_{j_0}=t_{j_0}e_1$ for some $j_0$. Comparing the first components of both sides of \eqref{eq h}, we get $h_{j_0,x_1}>0$. Thus, we have $h_j\in\{\xi\in\HH^n:x_1\geq0\}$ in \eqref{eq h}, for all $j=1,\ldots,M$.}
\end{remark}

\begin{lemma}\label{lemma 5.3}
Let $1<p<\infty$, $0<s<1$ and $0\leq\alpha<\min\{Q-2,p(1-s)\}$. For any $R_0>0$ and for $\beta>-\frac{1}{p-1}$
we have $u_{\beta,R_0}\in W^{s,p,\alpha}_{\mathrm{loc}(x_1)}(\HH^n_+)\cap L_{sp,\alpha}^{p-1}(\HH^n)$.
\end{lemma}

\begin{proof}
First, we show that $u_{\beta,R_0}\in W_{\mathrm{loc}(x_1)}^{s,p,\alpha}(\HH^n_+)$. Let $0<r<R$ and $\mathscr{D}_{r,R}\subset\HH^n_+$ be the set defined in \eqref{set D_r,R}. We have
\begin{multline}\label{eq 5.9}
[u_{\beta,R_0}]^p_{s,p,\alpha,\mathscr{D}_{r,R}}=\int_{\mathscr{D}_{r,R}}\int_{\mathscr{D}_{r,R}} \frac{|u_{\beta,R_0}(\xi)-u_{\beta,R_0}(\xi')|^p}{d({\xi}^{-1}\circ \xi')^{Q+sp}|z-z'|^\alpha}d\xi'd\xi\\
\leq \int_{\mathscr{D}_{r,R}}\int_{\mathscr{D}_{r,R}\cap B(\xi,1)} \frac{|u_{\beta,R_0}(\xi)-u_{\beta,R_0}(\xi')|^p}{d({\xi}^{-1}\circ \xi')^{Q+sp}|z-z'|^\alpha}d\xi'd\xi\\
+2^{p-1}\int_{\mathscr{D}_{r,R}}\int_{\mathscr{D}_{r,R}\cap B(\xi,1)^C} \frac{|u_{\beta,R_0}(\xi)|^p+|u_{\beta,R_0}(\xi')|^p}{d({\xi}^{-1}\circ \xi')^{Q+sp}|z-z'|^\alpha}d\xi'd\xi=:I_1+2^{p-1}I_2.
\end{multline}
The last integral can be written as
$$I_2=2\int_{\mathscr{D}_{r,R}}|u_{\beta,R_0}(\xi)|^p\int_{\mathscr{D}_{r,R}\cap B(\xi,1)^C} \frac{d\xi'}{d({\xi}^{-1}\circ \xi')^{Q+sp}|z-z'|^\alpha}d\xi.$$
This integral was shown to be finite in \cite{Adi-Mallick} (see $I_1$ in Proposition 3.1). It remains to prove $I_1<\infty$. By symmetry, $I_1$ can be written as
$$I_1=2\int_{\mathscr{D}_{r,R}}\int_{\{\xi'\in\,\mathscr{D}_{r,R}\cap B(\xi,1):x_1<x_1'\}} \frac{|u_{\beta,R_0}(\xi)-u_{\beta,R_0}(\xi')|^p}{d({\xi}^{-1}\circ \xi')^{Q+sp}|z-z'|^\alpha}d\xi'd\xi.$$
Changing the variable $\xi'$ by $\xi\circ\tilde{\xi}$, we have
\begin{equation}\label{eq 5.12}
I_1=2\int_{\mathscr{D}_{r,R}}\int_{\{\tilde{\xi}\in\,(\mathscr{D}_{r,R}-\xi)\cap B(0,1):\tilde{x}_1>0\}} \frac{|u_{\beta,R_0}(\xi)-u_{\beta,R_0}(\xi\circ\tilde{\xi})|^p}{d(\tilde{\xi})^{Q+sp}|\tilde{z}|^\alpha}d\tilde{\xi}d\xi,
\end{equation}
where $\mathscr{D}_{r,R}-\xi:=\{\xi^{-1}\circ\xi':\xi'\in\mathscr{D}_{r,R}\}$. By Theorem \ref{thm 5.2} and Remark \ref{remark 3}, we have
$$\tilde{\xi}=h_1\circ\ldots\circ h_M,$$
$$h_j\in\{\xi\in\HH^n:x_1\geq0\},\ \ d(h_j)\leq c_0 d(\tilde{\xi)},\ \ j=1,\ldots,M,$$
where $h_j=\mathrm{Exp}(t_jX_{i_j})$, for some $t_j\in\RR$ and $d(h_j)=|t_j|$. Proceeding similarly to $I_2$ in Proposition 3.1 in \cite{Adi-Mallick}, we get
$$|u_{\beta,R_0}(\xi)-u_{\beta,R_0}(\xi\circ\tilde{\xi})|^p\leq c_0 d(\tilde{\xi})\sum_{j=1}^M\int_0^1\left|\nabla_{\HH^n}^\xi u_{\beta,R_0}(\xi\circ h_1\circ\ldots\circ h_{j-1}\circ \mathrm{Exp}(t_jrX_{i_j})\right|^pdr$$
Substituting this in \eqref{eq 5.12}, we obtain
$$I_1\leq2c_0\sum_{j=1}^M\int_0^1\int_{\mathscr{D}_{r,R}}\left|\nabla_{\HH^n}^\xi u_{\beta,R_0}(\xi\circ h_1\circ\ldots\circ h_{j-1}\circ \mathrm{Exp}(t_jrX_{i_j})\right|^p\int_{B(0,1)} \frac{d\tilde{\xi}}{d(\tilde{\xi})^{Q+sp-p}|\tilde{z}|^\alpha}d\xi dr$$
The integral $\int_{B(0,1)} \frac{d\tilde{\xi}}{d(\tilde{\xi})^{Q+sp-p}|\tilde{z}|^\alpha}<\infty$ is shown in \cite{Adi-Mallick}  provided $\alpha<\min\{Q-2,p(1-s)\}$ (see integral $J$ in Proposition 3.1 in \cite{Adi-Mallick}), and $I_1<\infty$ follows from Lemma \ref{lemma 5.2}, since $\xi\circ h_1\circ\ldots\circ h_{j-1}\circ \mathrm{Exp}(t_jrX_{i_j})\\\in \mathscr{D}_{r,R}$ for all $j=1,\ldots,M$.\smallskip

Since $u_{\beta,R_0}\in L^\infty(\HH^n)$ for all $\beta\geq0$, it is enough to show $u_{\beta,R_0}\in L_{sp,\alpha}^{p-1}(\HH^n)$ only for $-\frac{1}{p-1}<\beta<0$. For such $\beta$, we have
\begin{multline*}
\int_{\HH^n}\frac{|u_{\beta,R_0}|^{p-1}}{(1+d(\xi))^{Q+sp}(1+|z|)^\alpha}d\xi\leq \int_{\HH^n_+}\frac{x_1^{\beta(p-1)}}{(1+(|z|^4+t^2)^\frac{1}{4})^{Q+sp}(1+|z|)^\alpha}dzdt\\
\leq \int_{\{(z,t)\in\HH^n_+:x_1>1\}}\frac{dzdt}{(1+(|z|^4+t^2)^\frac{1}{4})^{Q+sp}(1+|z|)^\alpha}\\
+\int_{R^{2n-1}}\int_{-\infty}^\infty\int_0^1\frac{x_1^{\beta(p-1)}}{(1+((x_1^2+|\tilde{x}|^2)^2+t^2)^\frac{1}{4})^{Q+sp}(1+(x_1^2+|\tilde{x}|^2)^\frac{1}{2})^\alpha}dx_1d\tilde{x}dt=J_1+J_2.
\end{multline*}
We clearly have
$$J_1\leq \int_{B(0,1)^C}\frac{dzdt}{(|z|^4+t^2)^\frac{Q+sp}{4}|z|^\alpha},$$
which is same as the integral $I_3$ in Proposition 3.1 in \cite{Adi-Mallick}. The integral $J_2$ can be estimated as follows:
\begin{multline}\label{eq 5.13}
J_2\leq \int_{R^{2n-1}}\int_{-\infty}^\infty\frac{d\tilde{x}dt}{(1+(|\tilde{x}|^4+t^2)^\frac{1}{4})^{Q+sp}(1+|\tilde{x}|)^\alpha}\int_0^1x_1^{\beta(p-1)}dx_1\\
\leq\int_{R^{2n-1}}\int_{-\infty}^\infty\frac{d\tilde{x}dt}{(1+|\tilde{x}|^4+t^2)^{\frac{Q+sp}{4}}(1+|\tilde{x}|^2)^{\frac{\alpha}{2}}}\int_0^1x_1^{\beta(p-1)}dx_1.
\end{multline}
The integral $\int_0^1x_1^{\beta(p-1)}dx_1$ is finite due to the choice of $\beta$. The remainder of the integral on the right-hand side of \eqref{eq 5.13} can be transformed into the integral \eqref{eq 5.7} of Lemma \ref{lemma 5.1} by applying the inequality $1+|\tilde{x}|^4\geq ((1/\sqrt{2})+|\tilde{x}/(2^\frac{1}{4})|^2)^2$ and performing a change of variables in $\tilde{x}$, respectively.
\end{proof}

For $0<\epsilon\ll1$ let us define the following set in $\HH^n$:
\begin{equation}
\mathcal{K}_\epsilon(\xi)=\left\{\xi'\in\HH^n: \min\left\{\frac{x_1}{1+\epsilon},(1+\epsilon)x_1\right\}\leq x_1'\leq\max\left\{\frac{x_1}{1+\epsilon},(1+\epsilon)x_1\right\}\right\}.
\end{equation}

\begin{theorem}\label{thm 5.3}
Let $1<p<\infty$, $0<s<1$, $0\leq\alpha<\min\{Q-2,p(1-s)\}$, and $R_0>0$. If 
$$-\frac{1}{p-1}<\beta<\frac{sp+\alpha}{p-1},$$
then the function $u_{\beta,R_0}$ is a weak supersolution of \eqref{eq L} for $D_{R_0}\cap\HH^n_+$, with $\lambda=C_{n,s,p,\alpha}\lambda(\beta,s',p)$, where $s'=s+\frac{\alpha}{p}<1$.
\end{theorem}
\begin{proof}
Let $0<\epsilon\ll1$. For any $\xi\in D_{R_0}$, we have
$$\int_{\HH^n\setminus \mathcal{K}_\epsilon(\xi)}\frac{J_p(u_{\beta,R_0}(\xi)-u_{\beta,R_0}(\xi'))}{d({\xi}^{-1}\circ \xi')^{Q+sp}|z-z'|^\alpha}d\xi'=\int_{\HH^n\setminus (\xi^{-1}\circ\mathcal{K}_\epsilon(\xi))}\frac{J_p(u_{\beta,R_0}(\xi)-u_{\beta,R_0}(\xi\circ\xi'))}{d(\xi')^{Q+sp}|z'|^\alpha}d\xi'$$
Note that the set
$$\xi^{-1}\circ\mathcal{K}_\epsilon(\xi)=\left\{\xi'\in\HH^n: \min\left\{\frac{-\epsilon x_1}{1+\epsilon},\epsilon x_1\right\}\leq x_1'\leq\max\left\{\frac{-\epsilon x_1}{1+\epsilon},\epsilon x_1\right\}\right\},$$
and when $\xi\in D_{R_0}$ is such that $x_1<0$, then $u_{\beta,R_0}(\xi)-u_{\beta,R_0}(\xi\circ\xi')=0$ for all $\xi'\in \xi^{-1}\circ\mathcal{K}_\epsilon(\xi)$. We also have
$$u_{\beta,R_0}(\xi)-u_{\beta,R_0}(\xi\circ\xi')\geq u_{\beta}(\xi)-u_{\beta}(\xi\circ\xi'),\ \ \forall \xi\in D_{R_0}.$$
Therefore, we can write
\begin{multline*}
\int_{\HH^n\setminus \mathcal{K}_\epsilon(\xi)}\frac{J_p(u_{\beta,R_0}(\xi)-u_{\beta,R_0}(\xi'))}{d({\xi}^{-1}\circ \xi')^{Q+sp}|z-z'|^\alpha}d\xi'\\
\geq\int_{\RR\setminus\mathcal{J}_\epsilon(x_1)}J_p(({x_1})_+^\beta-({x_1+x_1'})_+^\beta)\left(\int_{\RR^{2n-1}}\int_{-\infty}^\infty\frac{d\tilde{x}'dt'}{((x_1'^2+|\tilde{x}'|^2)^2+t'^2)^{\frac{Q+sp}{4}}(x_1'^2+|\tilde{x}'|)^{\frac{\alpha}{2}}}\right)dx_1',
\end{multline*}
where $\tilde{x}'=(x_2',\ldots,x_n',y_1',\ldots,y_n')\in\RR^{2n-1}$, and for $x_1>0$, $\mathcal{J}_\epsilon(x_1)$ is the interval $\left(\frac{-\epsilon x_1}{1+\epsilon},\epsilon x_1\right)$. Using Lemma \ref{lemma 5.1}, we get
\begin{multline}\label{eq 5.15}
\int_{\HH^n\setminus \mathcal{K}_\epsilon(\xi)}\frac{J_p(u_{\beta,R_0}(\xi)-u_{\beta,R_0}(\xi'))}{d({\xi}^{-1}\circ \xi')^{Q+sp}|z-z'|^\alpha}d\xi'\geq C_{n,s,p,\alpha}\int_{\RR\setminus\mathcal{J}_\epsilon(x_1)}\frac{J_p(({x_1})_+^\beta-({x_1+x_1'})_+^\beta)}{x_1'^{1+sp+\alpha}}dx_1'\\
=C_{n,s,p,\alpha}\int_{\RR\setminus\mathcal{I}_\epsilon(x_1)}\frac{J_p(({x_1})_+^\beta-({x_1'})_+^\beta)}{|x_1-x_1'|^{1+sp+\alpha}}dx_1'.
\end{multline}
The constant $C_{n,s,p,\alpha}$ and the set $\mathcal{I}_\epsilon(x_1)\subset\RR_+$ are defined in \eqref{eq C(n,s,p)} and \eqref{set I_ep}, respectively. Applying \eqref{eq F-ep} in \eqref{eq 5.15} with $s'=s+\frac{\alpha}{p}$, we obtain
$$2\int_{\HH^n\setminus \mathcal{K}_\epsilon(\xi)}\frac{J_p(u_{\beta,R_0}(\xi)-u_{\beta,R_0}(\xi'))}{d({\xi}^{-1}\circ \xi')^{Q+sp}|z-z'|^\alpha}d\xi'\geq C_{n,s,p,\alpha}\lambda_\epsilon(\beta,s',p)\frac{u_{\beta}(\xi)^{p-1}}{x_1^{sp+\alpha}}.$$
Let $0\leq\phi\in C_c^\infty(D_{R_0}\cap\HH^n_+)$. Multiply the above equation by $\phi(\xi)$ and then integrate over $\HH^n_+$, we get
\begin{equation}\label{eqn 5.15}
2\int_{\HH^n_+}\left(\int_{\HH^n\setminus \mathcal{K}_\epsilon(\xi)}\frac{J_p(u_{\beta,R_0}(\xi)-u_{\beta,R_0}(\xi'))}{d({\xi}^{-1}\circ \xi')^{Q+sp}|z-z'|^\alpha}d\xi'\right)\phi(\xi)d\xi\geq C_{n,s,p,\alpha}\lambda_\epsilon(\beta,s',p)\int_{\HH^n_+}\frac{u_{\beta}(\xi)^{p-1}}{x_1^{sp+\alpha}}\phi(\xi)d\xi.
\end{equation}
On the other hand, applying the Dominated Convergence Theorem (thanks to Lemma \ref{Lemma 2.1} and Lemma \ref{lemma 5.3}) we can write
\begin{multline*}
\int_{\HH^n}\int_{\HH^n}\frac{J_p(u_{\beta,R_0}(\xi)-u_{\beta,R_0}(\xi'))(\phi(\xi)-\phi(\xi'))}{d({\xi}^{-1}\circ \xi')^{Q+sp}|z-z'|^\alpha}d\xi'd\xi\\
=\lim_{\epsilon\rightarrow0}\iint_{(\HH^n\times\HH^n)\setminus\mathcal{C}_\epsilon}\frac{J_p(u_{\beta,R_0}(\xi)-u_{\beta,R_0}(\xi'))(\phi(\xi)-\phi(\xi'))}{d({\xi}^{-1}\circ \xi')^{Q+sp}|z-z'|^\alpha}d\xi'd\xi,
\end{multline*}
where
$$\mathcal{C}_\epsilon:=\{(\xi,\xi')\in\HH^n\times\HH^n:\xi'\in\mathcal{K}_\epsilon(\xi)\}.$$
In a similar way to the proof of Lemma \ref{lemma 3.1}, one can show that
$$\frac{J_p(u_{\beta,R_0}(\xi)-u_{\beta,R_0}(\xi'))}{d({\xi}^{-1}\circ \xi')^{Q+sp}|z-z'|^\alpha}\phi(\xi)\in L^1((\HH^n\times\HH^n)\setminus\mathcal{C}_\epsilon)\ \ \mathrm{for}\ 0<\epsilon\ll1.$$
Therefore, using Fubini's theorem, we can rewrite the above integral of the following form:
\begin{multline}\label{eq 5.16}
\int_{\HH^n}\int_{\HH^n}\frac{J_p(u_{\beta,R_0}(\xi)-u_{\beta,R_0}(\xi'))(\phi(\xi)-\phi(\xi'))}{d({\xi}^{-1}\circ \xi')^{Q+sp}|z-z'|^\alpha}d\xi'd\xi\\
=2\lim_{\epsilon\rightarrow0}\int_{\HH^n_+}\left(\int_{\HH^n\setminus \mathcal{K}_\epsilon(\xi)}\frac{J_p(u_{\beta,R_0}(\xi)-u_{\beta,R_0}(\xi'))}{d({\xi}^{-1}\circ \xi')^{Q+sp}|z-z'|^\alpha}d\xi'\right)\phi(\xi)d\xi.
\end{multline}
Combining \eqref{eqn 5.15} and \eqref{eq 5.16}, together with \eqref{eq beta}, we get the conclusion.
\end{proof}

Since we work with a cutoff function, let us define the following sharp constant for domains. For an open set $\Omega\subset\HH^n_+$, let us define the following quantity:
\begin{equation}
\mathcal{C}_{n,s,p,\alpha}(\Omega)=\underset{f\in C_c^\infty(\Omega)}{\inf}\left\{[f]_{s,p,\alpha,\HH^n}^p:\int_{\HH^n_+}\frac{|f(\xi)|^p}{x_1^{sp+\alpha}}d\xi=1\right\}.
\end{equation}
\begin{remark}
\emph{It is easy to see that $\mathcal{C}_{s,p,\alpha}(\HH^n_+)\leq \mathcal{C}_{n,s,p,\alpha}(\Omega)$.}
\end{remark}
\begin{lemma}\label{lemma 5.4}
Let $1<p<\infty$, $0<s<1$, $\alpha\geq0$, and $\Omega\subset\HH^n_+$ be an open set. If there exists $\lambda\geq0$ such that equation \eqref{eq L} admits a positive weak supersolution $f$, then $\lambda\leq \mathcal{C}_{n,s,p,\alpha}(\Omega)$.
\end{lemma}
\begin{proof}
Let $\eta\in C_c^\infty(\Omega)$, and for $\epsilon>0$ define the function
$$\phi=\frac{|\eta|^p}{(f+\epsilon)^{p-1}}.$$
Using Lemma \ref{lemma 3.2} and a similar approach to \eqref{eq 4.4}, one can observe that $\phi$ is an admissible test function. Since $f$ is a positive weak supersolution to \eqref{eq L}, we have
$$\lambda\int_{\HH^n_+}\frac{f^{p-1}}{x_1^{sp+\alpha}}\frac{|\eta|^p}{(f+\epsilon)^{p-1}}d\xi\leq\int_{\HH^n}\int_{\HH^n}\frac{J_p(f(\xi)-f(\xi'))\left(\frac{|\eta|^p}{(f+\epsilon)^{p-1}}(\xi)-\frac{|\eta|^p}{(f+\epsilon)^{p-1}}(\xi')\right)}{d({\xi}^{-1}\circ \xi')^{Q+sp}|z-z'|^\alpha}d\xi' d\xi.$$
Using the elementary inequality $|a-\tau|^p\geq(1-\tau)^{p-1}(|a|^p-\tau)$ for $p\geq1$, $0\leq\tau\leq1$ and $a\in\RR$ (the proof can be found in \cite[Lemma 2.6]{Frank}), we obtain
$$\lambda\int_{\HH^n_+}\frac{f^{p-1}}{x_1^{sp+\alpha}}\frac{|\eta|^p}{(f+\epsilon)^{p-1}}d\xi\leq\int_{\HH^n}\int_{\HH^n}\frac{\Big||\eta(\xi)|-|\eta(\xi')|\Big|^p}{d({\xi}^{-1}\circ \xi')^{Q+sp}|z-z'|^\alpha}d\xi' d\xi\leq [|\eta|]_{s,p,\alpha,\HH^n}\leq [\eta]_{s,p,\alpha,\HH^n}.$$
Here we consider $a=\frac{|\eta(\xi)|}{|\eta(\xi)|}\frac{f(\xi')+\epsilon}{f(\xi)+\epsilon}$; by symmetry we can assume $f(\xi')\leq f(\xi)$ and we consider $\tau=\frac{f(\xi')}{f(\xi)}$. Finally, we apply Fatou's lemma and pass the limit $\epsilon\rightarrow0$.
The arbitrariness of the function $\eta$ concludes the proof.
\end{proof}
Now we prove the sharp constant result Theorem \ref{thm sharp C} for $\HH^n_+$.
\begin{proof}[\textbf{Proof of Theorem \ref{thm sharp C}}]
Let $s'=s+\frac{\alpha}{p}$. Lemma \ref{lemma 5.4} together with Theorem \ref{thm 5.3} gives
$$\mathcal{C}_{n,s,p,\alpha}(D_{R_0}\cap\HH^n_+)\geq C_{n,s,p,\alpha}\lambda(\beta,s',p),$$
for every $R_0>0$ and $-\frac{1}{p-1}<\beta<\frac{sp+\alpha}{p-1}$. Moreover, by Proposition \ref{prop 5.1}, $\lambda(\beta,s',p)$ is maximal at $\beta=\frac{s'p-1}{p}$, and the maximum value is $\Lambda_{s,p,\alpha}$. Therefore, we get
\begin{equation}\label{eq 5.19}
\mathcal{C}_{n,s,p,\alpha}(D_{R_0}\cap\HH^n_+)\geq C_{n,s,p,\alpha}\Lambda_{s,p,\alpha}.
\end{equation}\smallskip

For the reverse inequality, let us choose a test function in such a way that it can be reduced to the Euclidean one-dimensional case by applying Lemma \ref{lemma 5.1}.\smallskip

\noindent Let $\eta\in C_c^\infty((0,\infty))$. Choose $R_0>0$ such that $\mathrm{supp}\,\phi\subset(0,R_0)$. Let $\Psi\in C_c^\infty((-1,2)\times(-1,1)^{2n})$ be such that $0\leq\Psi\leq1$, $\Psi$ is increasing in $(-1,0)$ and decreasing in (1,2) with respect to $x_1$ variable, $\Psi(x_1,\tilde{x},t)=\Psi(\tilde{x},t)$ for $0\leq x_1\leq1$, where $\tilde{x}=(x_2,\ldots,x_n,y_1,\ldots,y_n)\in\RR^{2n-1}$, $t\in\RR$ and $\xi=(x_1,\tilde{x},t)$. We define the test function
$$\phi(\xi)=\eta(x_1)\Psi_{R_0}(\xi),$$
where $\Psi_{R_0}(\xi)=\frac{R_0^{-\frac{2n+1}{p}}}{\lVert\Psi(0,\cdot)\rVert_{L^p(\RR^{2n})}}\Psi(\delta_{\frac{1}{R_0}}\xi)$. It is easy to see that $\phi$ has compact support in $D_{R_0}\cap\HH^n_+$ and $\lVert\Psi_{R_0}(0,\cdot)\rVert_{L^p(\RR^{2n})}=1$. From the definition of $\mathcal{C}_{n,s,p,\alpha}(D_{R_0}\cap\HH^n_+)$, we get
$$\mathcal{C}_{n,s,p,\alpha}(D_{R_0}\cap\HH^n_+)\leq\frac{[\eta\Psi_{R_0}]_{s,p,\alpha,\HH^n}^p}{\int_{\HH^n_+}\frac{|\eta\Psi_{R_0}|^p}{x_1^{sp+\alpha}}d\xi}=\frac{[\eta\Psi_{R_0}]_{s,p,\alpha,\HH^n}^p}{\int_0^{R_0}\frac{|\eta(x_1)|^p}{x_1^{sp+\alpha}}dx_1\int_{\RR^{2n}}|\Psi_{R_0}(0,\tilde{x},t)|^pd\tilde{x}dt}=\frac{[\eta\Psi_{R_0}]_{s,p,\alpha,\HH^n}^p}{\int_0^\infty\frac{|\eta(x_1)|^p}{x_1^{sp+\alpha}}dx_1}.$$
Using Minkowski's inequality, we split the integral
\begin{multline*}
[\eta\Psi_{R_0}]_{s,p,\alpha,\HH^n}\\
=\left(\int_{\HH^n}\int_{\HH^n}\frac{|\eta(x_1)\Psi_{R_0}(\xi)-\eta(x_1')\Psi_{R_0}(\xi)+\eta(x_1')\Psi_{R_0}(\xi)-\eta(x_1')\Psi_{R_0}(\xi')|^p}{d({\xi}^{-1}\circ \xi')^{Q+sp}|z-z'|^\alpha}d\xi'd\xi\right)^{\frac{1}{p}}\\
\leq \left(\int_{\HH^n}\int_{\HH^n}\frac{|\Psi_{R_0}(\xi)|^p|\eta(x_1)-\eta(x_1')|^p}{d({\xi}^{-1}\circ \xi')^{Q+sp}|z-z'|^\alpha}d\xi'd\xi\right)^{\frac{1}{p}}\\
+\left(\int_{\HH^n}\int_{\HH^n}\frac{|\eta(x_1')|^p|\Psi_{R_0}(\xi)-\Psi_{R_0}(\xi')|^p}{d({\xi}^{-1}\circ \xi')^{Q+sp}|z-z'|^\alpha}d\xi'd\xi\right)^{\frac{1}{p}}=:I_1^{\frac{1}{p}}+I_2^{\frac{1}{p}}.
\end{multline*}
Applying Fubini's theorem and Lemma \ref{lemma 5.1}, we can estimate the integral $I_1$ as follows:
\begin{multline*}
I_1=\int_{\HH^n}\int_{\HH^n}\frac{|\Psi_{R_0}(\xi)|^p|\eta(x_1)-\eta(x_1+x_1')|^p}{d(\xi')^{Q+sp}|z'|^\alpha}d\xi'd\xi\\
\leq C_{n,s,p,\alpha}\int_{\RR^{2n}}|\Psi_{R_0}(0,\tilde{x},t)|^p\int_\RR\int_\RR\frac{|\eta(x_1)-\eta(x_1+x_1')|^p}{|x_1'|^{1+sp+\alpha}}dx_1'dx_1=C_{n,s,p,\alpha}[\eta]_{W^{s',p}(R)}^p.
\end{multline*}
Here we use the monotonic property of $\Psi$. Now we estimate the integral $I_2$ as follows:
$$I_2\leq \lVert\eta\rVert_{L^\infty(\RR)}^p\int_{\HH^n}\int_{\HH^n}\frac{|\Psi_{R_0}(\xi)-\Psi_{R_0}(\xi')|^p}{d({\xi}^{-1}\circ \xi')^{Q+sp}|z-z'|^\alpha}d\xi'd\xi.$$
Applying the changes of variable $\xi=\delta_{R_0}\Bar{\xi}$ and $\xi'=\delta_{R_0}\Bar{\xi}'$, we obtain
$$I_2\leq \lVert\eta\rVert_{L^\infty(\RR)}^p \frac{R_0^{-(2n+1)}}{\lVert\Psi(0,\cdot)\rVert_{L^p(\RR^{2n})}^p}\frac{R_0^{2Q}}{R_0^{Q+sp+\alpha}}[\Psi]_{s,p,\alpha,\HH^n}^p=CR_0^{1-sp-\alpha}.$$
Here the constant $C$ is independent to $R_0$. Therefore, we get
$$\mathcal{C}_{n,s,p,\alpha}(D_{R_0}\cap\HH^n_+)^{\frac{1}{p}}\leq
\frac{C_{n,s,p,\alpha}^{\frac{1}{p}}[\eta]_{W^{s',p}(R)}}{\left(\int_0^\infty\frac{|\eta(x_1)|^p}{x_1^{sp+\alpha}}dx_1\right)^{\frac{1}{p}}}+\frac{C^{\frac{1}{p}}R_0^{\frac{1-sp-\alpha}{p}}}{\left(\int_0^\infty\frac{|\eta(x_1)|^p}{x_1^{sp+\alpha}}dx_1\right)^{\frac{1}{p}}}.$$
Since $sp+\alpha>1$ and $\eta$ is arbitrary, using the one-dimensional Euclidean result \eqref{sharp c}, we obtain
\begin{equation}\label{eq 5.20}
\lim_{R_0\rightarrow\infty}\mathcal{C}_{n,s,p,\alpha}(D_{R_0}\cap\HH^n_+)\leq C_{n,s,p,\alpha}\Lambda_{s,p,\alpha}.
\end{equation}
The final conclusion follows from \eqref{eq 5.19} and \eqref{eq 5.20}, and from the following fact:
$$\lim_{R_0\rightarrow\infty}\mathcal{C}_{n,s,p,\alpha}(D_{R_0}\cap\HH^n_+)=\mathcal{C}_{s,p,\alpha}(\HH^n_+).$$

\end{proof}

\noindent\textbf{Acknowledgement:} The author thanks Prof. Rama Rawat for several discussions on this problem. He acknowledges the support of IIT Kanpur, India, for the Fellowship for Academic and Research Excellence (FARE).

\end{document}